\documentclass[10pt,english]{article}

\usepackage[utf8]{inputenc}

\usepackage{csquotes}
\usepackage[style=numeric,giveninits=true]{biblatex}
\usepackage[a4paper]{geometry}

\usepackage{comment}
\usepackage{mathtools}
\usepackage{amsmath}
\usepackage{amssymb}
\usepackage{mathrsfs}
\usepackage{amsthm} 
\usepackage{enumerate}
\usepackage{tikz}
\usepackage{graphicx}

\usepackage{xcolor}
\usepackage{authblk} 
\usepackage{babel}
\usepackage{hyperref}

\newcommand{\RR}{{\mathbb R}}
\newcommand{\NN}{{\mathbf N}}
\newcommand{\comp}[1]{{#1^{\mathrm{c}}}} 
\newcommand{\vit}{{\mathcal V}} 
\newcommand{\acc}{{\mathcal A}} 
\newcommand{\pp}[1]{\prescript{\perp}{}{#1}} 
\newcommand{\potlmt}{\varphi} 
\newcommand{\one}{{\mathbf 1}} 
\newcommand{\proba}{{\mathcal P}} 
\newcommand{\cont}{C} 
\newcommand{\contcomp}{\cont_{\mathrm{c}}} 
\newcommand{\contbound}{\cont_{\mathrm{b}}} 
\newcommand{\eps}{\varepsilon}
\newcommand{\dd}{{\mathrm d}} 
\newcommand{\funceps}[1]{#1^{\eps}} 
\newcommand{\funcepsn}[2]{#1^{\eps_{#2}}} 
\newcommand{\chgvar}[1]{\tilde{#1}} 
\newcommand{\rot}[1]{\mathcal{R} \left ( #1 \right )} 
\newcommand{\wass}{W} 
\newcommand{\uvec}{e_3} 
\newcommand{\adim}[1]{\hat{#1}} 
\newcommand{\transpt}{{\mathcal T}} 
\newcommand{\ttranspt}{{\tilde{\mathcal T}}}
\newcommand{\pushfwd}[1]{#1 \#} 
\newcommand{\kervit}{K^{\mathcal{V}}} 
\newcommand{\keracc}{K^{\mathcal{A}}} 
\newcommand{\conv}{\ast} 
\newcommand{\ball}{B} 
\newcommand{\bigO}{\mathcal{O}}

\DeclareMathOperator{\esssup}{ess \, sup} 
\DeclareMathOperator{\divg}{div} 
\DeclareMathOperator{\supp}{Supp} 
\DeclareMathOperator{\dist}{dist} 

\newcommand{\var}{0.2} 
\newcommand{\bd}{2.5} 

\newtheorem{theorem}{Theorem}
\newtheorem{definition}{Definition}

\newtheorem{prop}{Proposition}[section]
\newtheorem{lemma}{Lemma}[section]

\title{Well-posedness of a 2D gyrokinetic model with equal Debye length and Larmor radius}
\author[1]{Pierre-Antoine Giorgi}
\author[2]{Maxime Hauray}
\affil[2]{Aix Marseille Univ, CNRS, Centrale Marseille, I2M, Marseille, France}
\affil[1]{Université de la Polynésie Française}
\date{\today}

\begin{document}
\maketitle

\begin{abstract}
We study here a 2D gyrokinetic model obtained in~\autocite{BOS2016-short}, which naturally appears as the limit of a Vlasov-Poisson system with a very large external uniform magnetic field in the finite Larmor radius regime, when the typical Larmor radius is of order of the Debye length.
We show that the Cauchy problem for that system is well-posed in a suitable space, provided that the initial condition satisfies a standard uniform decay assumption in velocity.  Our result relies on a  stability estimate in Wasserstein distance of order one between two solutions of the system.  That stability estimate directly implies the uniqueness (in an appropriate space) of solution to the Cauchy problem. An extension of the stability estimate to the case of a regularized interaction allows to prove the existence of solutions, as limits of solutions of a similar system with regularized interactions.
\end{abstract}

\section{Introduction}

\subsection{The origin of the model under investigation}
The equation considered in this article is a limit equation of the 2D Vlasov-Poisson system with a large external uniform magnetic field, in the finite Larmor radius regime. The finite Larmor radius regime means that while the strength of the magnetic field is sent to infinity, the spatial scale in the direction perpendicular to the magnetic field is set appropriately in such a way that the fast gyration of the charged particle under investigation is performed with a finite Larmor radius\footnote{if no spatial rescaling is done, the Larmor radius goes to zero and this regime is usually  called the "guiding center regime".}.

This regime is commonly used in the physics of fusion plasmas  because it is a relevant approximation to model and simulate the core of a tokamak~\cite{Grandgirard}, where knowledge of the particle distribution function at the scale of the typical Larmor radius is important. Usually, the Debye length remains much smaller than the typical Larmor radius, so that at the Larmor scale the plasma is quasi-neutral. It means that at Larmor scale, the gyrokinetic equation for the distribution function of the gyrocenter is coupled with an electro-neutrality equation. 

But from a mathematical viewpoint, quasi-neutral Vlasov equation are very difficult to study. Their associated Cauchy problem are in full generality ill-posed~\cite{BAR2012}. Recently \citeauthor{HK2016} showed in a deep paper~\cite{HK2016} that with Penrose stable~\cite{Penrose} initial data, a rigorous limit from Vlasov-Poisson towards a quasi-neutral Vlasov equation is possible. However, the combination of that quasi-neutral limit and the large magnetic field limit (with the appropriate ordering : the Debye length much smaller than the typical Larmor radius) seems a very difficult mathematical challenge.

Another regime, physically less pertinent but mathematically simpler has been investigated since roughly twenty years:  the limit of the 2D Vlasov-Poisson equation with a large magnetic field in the finite Larmor radius regime, when the typical Larmor radius is of the same order than the Debye length. It has been originally studied mathematically by E. Frénod and E. Sonn\"endrucker \cite{FRE1998,FRE2001,FRE2010}), and then by Bostan~\cite{BOS2009} and Han-Kwan~\cite{HK2012}. In the above mentioned works, the limit is a system of PDE in which the advection field is obtained by a two scale homogenization procedure, and is not fully explicit.

Recently, a simpler description of the limit was presented by~\citeauthor{BOS2016-short} in~\autocite{BOS2016-short}, with the rigorous proof of the convergence that can be found in the companion paper~\cite{BOS2016}. The idea was to average directly the fast gyration in the two particle interaction kernel, rather than in the Vlasov equation and the Poisson equation independently. In short to average the relative motion of the particle that creates the electric field and of the particle that is submitted to the field together. 
This idea leads to an evolution equation on  $f$ the density distribution of the gyrocenter (center of fast circle of gyration) $x \in \RR^2$, and velocities $v = \rho e^{i \theta} \in \RR^2$ identified with $\mathbb C$, where $\rho \in \RR$ is the Larmor radius of the particle and $\theta \in \RR$ is an important phase parameter. We emphasize that $x$ and $v$ are not the position and velocity of a particle. It means that the particles are moving (infinitely) fast on the gyrocircle of center $x$ and radius $\rho$. Since the very fast rotation of all the particles is performed \emph{at the same speed}\footnote{The speed of rotation depends on the strength of the magnetic field that is considered uniform here.}, no homogenisation will hold on these circles, and it is crucial to keep track of the relative position of the particles on the circle. This explains why the limit model still contains a full $v$ variable and not only the Larmor radius~$\rho$.

\begin{figure}[ht]
    \centering
    \includegraphics[width=.7\textwidth]{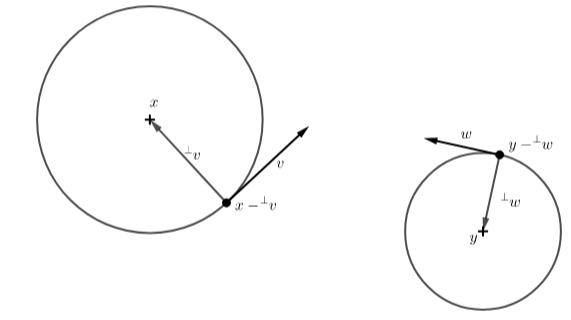}
    \caption{Two fast rotating particles around their gyrocenters. As they rotate at same speed, their relative motion is still circular. For simplicity, we assume as in the rest of the article, that the Larmor pulsation $\omega_c = \frac{|q|B}{2m} =1$.}
    \label{fig:gyrocentres}
\end{figure}

\paragraph{Motivation and interest of our result} 

We briefly recall how \citeauthor{BOS2016} \autocite{BOS2016} obtained the limit equation~\eqref{limit-equation}. Their results are summarized in~\autocite{BOS2016-short}.
Given an external uniform magnetic field $B$, they consider the Vlasov-Poisson system for one species of charged particles in two dimensions, in the plane perpendicular to the magnetic field $B$, which reads after proper scaling:
\begin{align*}
	\partial_{t} \funceps{f} + \frac{1}{\eps} \left ( v \cdot \nabla_{x} \funceps{f} + \omega_c \pp{v} \cdot \nabla_{v} \funceps{f} \right ) - \nabla_{x} \funceps{\phi} \cdot \nabla_{v} \funceps{f} &= 0, \\
	- \Delta_x \funceps{\phi} &= \int_{\RR^2} \funceps{f} \, \dd v,
\end{align*}
for $t>0$ and $(x,v) \in \RR^2 \times \RR^2$.
The number $\eps$ is the ratio of the Debye length in the direction perpendicular to the magnetic field to the observation scale\footnote{the usual 3D formula for the Debye length should not be used here, where the scaling in position depends on the direction. Here the Debye length in the perpendicular direction should be defined - as usual in fact - as the scale of the electric phenomena in the directions perpendicular to the magnetic field.}.
$\omega_c = \frac{q |B|}{2m}$ is the Larmor pulsation. In the sequel, we will assume that it is equal to $1$. It has no incidence on the mathematical analysis, and could physically be done with the help of the ad-hoc change of time scale\footnote{at the price of introducing a different factor in the equation for the potential.}.  

A large magnetic field induce a fast gyration of the charged particles around the magnetic lines. For that reason,  they perform the change of variables
\[
	\chgvar{x} = x + \pp{v}, \qquad 
	\chgvar{v} = \rot{\frac{t}{\eps}}v,
\]
which is done to filter out the fast circular motion that occurs at the scale of the cyclotronic period.
$\rot{\theta}$ denotes the rotation operator of angle $\theta$ in $\RR^2.$
The distribution of the couple "gyrocenter" and "filtered velocity" $\funceps{\chgvar{f}}$ defined by
\[
\funceps{\chgvar{f}} (t,\chgvar{x},\chgvar{v}) = \funceps{f} \left (t, \chgvar{x} - \rot{-\frac{t}{\eps}}\pp{\chgvar{v}}, \rot{-\frac{t}{\eps}} \chgvar{v} \right )
\]
satisfies the equation
\[
	\partial_t \funceps{\chgvar{f}} - \pp{\nabla}_x \funceps{\phi} \cdot \nabla_{\chgvar{x}} \funceps{\chgvar{f}} - \rot{\frac{t}{\eps}} \nabla_x \funceps{\phi} \cdot \nabla_{\chgvar{v}} \funceps{\chgvar{f}} = 0.
\]
\citeauthor{BOS2016} proved, using compactness methods, that there exists a sequence $(\eps_n)_n$ converging to $0$ such that $( \funcepsn{\chgvar{f}}{n} )$ converges strongly in $L^2([0,T] \times \RR^2 \times \RR^2)$, for all $T>0$, towards a solution $f$ of equation~\eqref{limit-equation}.

Here we shall provide a stability result on the limit equation that also implies the uniqueness of solution in an appropriate class. However, we cannot conclude from our result the convergence of the full sequence $(f^\eps)_{\eps >0}$ above, because in~\cite{BOS2016} they provide a global $L^2$ bound on the limit $f$ while we need a kind of weighted uniform control (see Definition~\ref{deft-norm-gamma}) to conclude to the uniqueness. 

Our second result is an existence result, based on a different strategy than in~\cite{BOS2016}: we mollify the equation and use a stability estimate to prove the convergence of solutions of this mollified equation to a solution of~\eqref{limit-equation} as the mollification parameter goes to zero.


\subsection{Precise form of the equation}
Using the above idea, the study leads to a limit equation that contains an explicit mean-field advection term (from now on we will write the time variable in subscript for convenience):
\begin{equation}\label{limit-equation}
	\partial_t f_t + \vit[f_t] \cdot \nabla_x f_t + \acc[f_t] \cdot \nabla_v f_t = 0, \quad t>0, \quad (x,v) \in \RR^2 \times \RR^2.
\end{equation}
The velocity (resp. acceleration) field $\vit$ (resp. $\acc$) are rotated gradient in $x$ (resp. in $v$)  
of the potential $\potlmt$
\begin{equation} \label{relation-vit-acc-pot}
	\vit[f] = - \pp{\nabla}_x \potlmt[f], \qquad 
	\acc[f] = \pp{\nabla}_v \potlmt[f],
\end{equation}
\[
	\potlmt[f_t](x,v) = \bigl[K * f_t \bigr] (x,v)  = - \frac{1}{2 \pi} \int_{\RR^4} \ln \bigl(|v-w| \vee  |x-y| \bigr)  f_t(y,w) \, \dd y \dd w,
\]
where $\pp{u}$ denotes the perpendicular vector $\pp{u} = (u_2, -u_1)$ for all $u=(u_1, u_2) \in \RR^2$.
The two particles potential $K$ is symmetric in $(x,v)$ and defined by\footnote{$a \vee b$ stand for the maximum among the two real numbers $a$ and $b$.}
\begin{equation} \label{eq:def_pot}
K(x,v) = - \frac{1}{2 \pi}  \ln \bigl( |x| \vee  |v| \bigr).
\end{equation}
The advection fields could also be written directly using the kernel 
\begin{equation} \label{eq:def_J}
J(x,v) = \frac{\pp{x}}{2 \pi |x|^2}\one_{|x| \ge |v|}.
\end{equation}
Remark that the derivative of $K$ are respectively $\nabla_x K(x,v) =J(x,v)$ and its symmetric expression $\nabla_v K(x,v) =J(v,x) = J \circ S (x,v)$, where  $S$ denotes the permutation $S(x,v) = (v,x)$. So the fields $\vit$ and $\acc$ can be rewritten as
\begin{align}
	\vit[f(t)](x,v) &= 
	\bigl[J * f\bigr] (x,v)= \frac{1}{2 \pi} \int_{\RR^4} J(x-y,v-w) f(\dd y,\dd w) ,\label{vit-expr} \\
	\acc[f(t)](x,v) &= 
	\bigl[(J \circ S) * f\bigr] (x,v)= \frac{1}{2 \pi} \int_{\RR^4} J(v-w,x-y) f(\dd y,\dd w).
	\label{acc-expr}
\end{align}
We will mostly use the latter expressions. Given the symmetry in positions and velocities between these two expressions, generally in the following  we will prove the results only for the velocity field $\vit$, the proof for $\acc$ being analogous, mostly up to a permutation of the variables $(x,v)$. 

\paragraph{Some explanation on the origin of that new potential}

We recall that a gyrocenter at position $x$ and ``velocity'' $v$ exactly means that the associated particle is rotating infinitely rapidly on the circle of center $x$ and radius $|v|$ (with a phase  $\theta$ defined by $v = |v| e^{i \theta}$). The value of the potential $K(x,v)$ is exactly the electric potential induced by that fast gyrating particle on a test particle at position $0$.
That formula arises from standard electrostatic considerations: 
\begin{itemize}
    \item when $|x| \ge |v|$, $0$ is outside the circle and the induced potential is equivalent to the repulsive one created by a fixed charge at $x$.
    \item when $|x| < |v|$, $0$ is inside the circle and the induced electric field should vanish, and the potential is constant in $x$.
\end{itemize}

This is also relevant in the case where two rotating particles interact. In order to simplify the explanation, we identify here $\RR^2$ with $\mathbb C$ and consider $x,v \in \mathbb C$. If the first particle rotate around position $x$, starting from position $x- \pp{v}$, then its movement along time is $X(t) = x - \pp{v} e^{i \alpha t}$, with $\alpha$ very large. For the second particle, rotating rapidly around $y$ with starting position $y+w$, then the movement is $Y(t) = y -\pp{w} e^{i \alpha t}$ with the same $\alpha$. So the relative movement is  $Y(t) - X(t) = y-x - \pp{(w-v)}e^{i \alpha t}$ and the ``relative'' particle behave exactly like she had a gyrocenter at position $y-x$ and velocity $w-v$. So the average electric potential created by the second particle on the first will be given by  $K(y-x,w-v)$.

\subsection{Notation and rigorous definition of solutions} \label{subsec:not}

We denote by $\proba(\RR^4)$ the set of probability measures on $\RR^4$, and $\proba_1(\RR^4)$ the subset of $\proba(\RR^4)$ composed of probability measures with finite first order moment, namely
\[
\int_{\RR^{4}} (|x|+|v|) \, \mu(\dd x, \dd v) < +\infty.
\]
We recall that the time variable $t$ is written in subscript for commodity reasons. 

\begin{definition} \label{def:sol}
	 We will say that $f \in L^1_{loc}\bigl(\RR_+;(L^1 \cap L^\infty)(\RR^4)\bigr)$ is a weak solution to~\eqref{limit-equation} with a given initial datum $f_0 \in \proba(\RR^4)$ if
for all $\psi \in \contcomp^1(\RR_+\times\RR^4)$,
	\begin{align}
		&\int_{\RR_+ \times \RR^4} f(t,x,v) \bigl(\partial_t + \vit[f(t)](x,v) \cdot \nabla_x + \acc[f(t)](x,v) \cdot \nabla_v \bigr) \psi (t,x,v) \, \dd t \dd x \dd v \nonumber \\
		&= - \int_{\RR^4} f_0(x,v) \psi(0,x,v) \, \dd x \dd v, \label{weak-sense-limit-equation}
	\end{align}
	where $\vit[f(t)]$ and $\acc[f(t)]$ is given by~\eqref{vit-expr} and~\eqref{acc-expr}.

\end{definition}

The requirements that $f_t \in L^1\cap L^\infty$ for (almost) all times $t \ge 0$ implies that the advection fields $\vit[f_t],\acc[f_t]$ are bounded for any time (See Proposition~\ref{bounded-fields} below). In that case the integral appearing in the above definition is correctly defined.
We next define very useful norms related to the decays of distribution at infinity.
\begin{definition}\label{deft-norm-gamma}
	For any $\gamma>0$ and any $f \in L^{\infty}(\RR^4)$, define the following norm :
	\[
		\|f\|_{\gamma} = \underset{(x,v) \in \RR^2\times\RR^2}{\esssup} (1+|x|)^{\gamma}(1+|v|)^{\gamma}|f(x,v)| \in [0,+\infty].
	\]
	We denote by $L^\infty_\gamma(\RR^4)$ the associated Banach space of function :
	\[
	L^\infty_\gamma(\RR^4) = \left\{ 
	f \in L^\infty(\RR^4) ; \| f \|_\gamma < \infty
	\right\}.
	\]
\end{definition}

We have the quite simple estimates (stated without proofs)
\begin{equation}\label{eq:prop_norm_gamma}
    \begin{aligned}
    \|f \|_\infty & \le \| f\|_\gamma, \\
    \text{and for } \gamma > 2, \; \| f\|_1 & \le \kappa_\gamma^2  \|f \|_\gamma \quad\text{with }
    \kappa_\gamma := \int_{\RR^2} \frac{\dd y}{(1+|y|)^\gamma}.
    \end{aligned}
\end{equation}

\subsection{Main results}

Our first result is a stability result in the Wasserstein metric $W_1$ for solutions of~\eqref{limit-equation} that decrease sufficiently fast at infinity in velocity and position: they should belong to $L^1_{\mathrm{loc}}\bigl(\RR^+,L^\infty_\gamma(\RR^4)\bigr)$ for some $\gamma >2$.

\begin{theorem}[Wasserstein stability]
\label{theorem-stability}
    Let $\gamma >2$. There exists a numerical constant $C_\gamma$ such that for $f_t, \, g_t \in L^1_{\mathrm{loc}}(\RR^+,(\proba_1 \cap L^\infty_\gamma)(\RR^4))$  weak solutions to \eqref{limit-equation}, then for all $t \ge 0$ 
	\begin{equation}\label{stability-estimate}
		\wass_1(f_t,g_t) \leq \wass_1(f_0,g_0) 
		e^{C_\gamma \int_0^t ( \|f_s\|_\gamma + \|g_s\|_\gamma ) \dd s},
	\end{equation}
	where $\wass_1$ denotes the Wasserstein distance of order $1$.
\end{theorem}
That stability result implies the uniqueness of the solution in the $L^\infty_\gamma$ class.
An assumption that is not very demanding as we will show that such norms are in fact propagated for classical solutions.  

Next using a classical approximation procedure, we will show the following existence result.

\begin{theorem} [Existence of global weak solutions] \label{theorem-existence}
	Let $f_0 \in \proba_1(\RR^4) \cap L^\infty_\gamma(\RR^4)$ for some $\gamma>2$. Then, there exists a global weak solution $f_t$ to \eqref{limit-equation} with initial condition $f_0$,
	that satisfies:
	\begin{enumerate}[i)]
	\item Conservation of $L^p$ norms for any $p \in [1,\infty]$: $ \forall \, t \ge 0, \; \|f_t\|_p =\|f_0\|_p$,
	\item Propagation of $L^\infty_\gamma$ norms: $ \displaystyle 
			\|f_t\|_{\gamma} \leq \left(1+c \|f_0\|_\infty^{\frac14} t \right)^{2\gamma} \|f_0\|_{\gamma}, \quad t \geq 0$, with a numerical constant $c$ defined in Proposition~\ref{bounded-fields}.
	\end{enumerate}
\end{theorem}
Both Theorems implies that the Cauchy Problem for~\eqref{limit-equation} is well posed for initial conditions in $\proba_1 \cap L^\infty_\gamma$ ($\gamma >2$), in the class  $L^1_{\mathrm{loc}}(\RR^+,(\proba_1 \cap L^\infty_\gamma)(\RR^4))$. 

\paragraph{Difference with the classical similar results for the original Vlasov-Poisson system}

The general idea is that the fast gyration of particles provides a kind of regularization of the interaction potential. It can be seen in the fact that the force kernel $J$ defined in~\eqref{eq:def_J} contains an extra indicator function w.r.t. the classical Poisson force kernel in dimension two :
\[
	J_P(x-y)  =  \frac{x-y}{2 \pi |x-y|^2}. 
\]
The classical kernel depends only on the position. The new dependence on $x$ and $v$ complicates a bit the structure of the vector fields but it also simplifies the problem raised by the singularity of the Poisson potential at $x=0$ (and $v$ arbitrary).
In fact $J$ has a point singularity at $x=0,v=0$ in the phase space $\RR^4$, while $J_P$ has the whole plane $\{x=0\}$ of singular points in $\RR^4$.

For that reason, it is quite natural to expect that the control of the singularity will be easier and that we should get simpler estimates than in the Vlasov-Poisson case. This is exactly what will happen. 
Let us recall briefly the classical well-posedness result for the original 2D Vlasov-Poisson system 
\[
\partial_t f + v \cdot \nabla_x f + (J_P \conv \rho_f) \cdot \nabla_v f = 0.
\]
In~\autocite{LOE2006}, \citeauthor{LOE2006} proved 
a stability result for weak solutions such that the spatial density belongs to $L^{\infty}_{\mathrm{loc}}(\RR^+, L^\infty(\RR^3))$ which implies uniqueness in that class\footnote{The uniqueness part of that result was extended in~\autocite{MIO2016} for weak solutions such that the $L^p$ norms of the density grow at most linearly with respect to $p$.
Uniqueness is also established in~\autocite{HOL2018} when the density belongs to a certain class of Orlicz spaces.}. Roughly his estimate  between two solutions $f$ and $g$ looks like the following
\[
\frac{d}{dt} W_1(f_t,g_t) \le 
C \bigl( \| \rho_{f_t}\|_\infty + \| \rho_{g_t}\|_\infty \bigr) 
W_1(f_t,g_t) \left( 1+ \ln^- W_1(f_t,g_t) \right),
\]
where $\ln^-$ denotes the negative part of $\ln$:  $\ln^-(x) = - [(\ln x) \wedge 0]$.  $\rho_f$ denotes the spatial density associated to $f$: $\rho_f = \int f \, \dd v$.

Then, an issue is to construct solutions to the Vlasov-Poisson system whose spatial density is bounded. This is not so easy. An approach based on velocity moments of the distribution function was initiated by~\citeauthor{LIO1991} in~\autocite{LIO1991} (see also later~\autocite{SAL2009,PAL2012}). An approach based originally on the preservation of the support of compactly supported function was initiated by~\citeauthor{Pfaff} in~\autocite{Pfaff}, and later extended to functions vanishing uniformly sufficiently fast in velocity at infinity  by~\citeauthor{Horst} in~\autocite{Horst}.

In our stability estimate~\eqref{stability-estimate}for our new model~\eqref{limit-equation}, we see several simplifications with respect to the VP case:
\begin{itemize}
\item Our stability estimate is linear rather than log-linear. This is a consequence of the screening.
\item It holds  on the condition that $\| f\|_\gamma$ and $\|g\|_\gamma$ are finite only. Such norms are much simpler to propagate than the infinite norms of spatial densities requested in the VP case. This simplification is also related to the screening, which allows to integrate on the velocity variable more easily since the interaction kernel vanishes at large velocities. 
\end{itemize}

\paragraph{Plan of the article}
Section~\ref{preliminaries} is dedicated to some preliminary results about equation~\eqref{limit-equation}.
The proof of the stability estimate~\eqref{stability-estimate} of Theorem~\ref{theorem-stability} is to be found in section~\ref{stability-wasserstein}. Finally, the existence of weak solutions (Theorem~\ref{theorem-existence}) is proved in section~\ref{existence-of-weak-solutions}.

\section{Preliminaries}\label{preliminaries}

In this section, we state some useful lemmas. Then we state and prove some a priori estimates on solutions to Equation~\eqref{limit-equation}.
At the end, we recall some facts about the Wasserstein distance of order $1$ that is used in the proof of Theorem~\ref{theorem-existence} and~\ref{theorem-stability}.

\subsection{Useful geometric lemmas}

We will state here four useful lemmas. The two first are about the way to control the increments of $J$.
The two last provide bound for some specific integrals, that will appears many time later in the proof. 

\begin{lemma}\label{lem:ineq_R2}
    Let $x,x_* \in \RR^2\setminus \{0\}$. Then
\[
	\left | \frac{\pp{x}}{|x|^2}-\frac{\pp{x_*}}{|x_*|^2} \right | \leq  
	\left(\frac{1}{|x|^2}+\frac{1}{|x_*|^2} \right ) \frac{|x-x_*|}2.
\]
\end{lemma}

\begin{proof}
  Indeed, notice that
\[
	\left | \frac{|x_*|\pp{x}}{|x|}-\frac{|x|\pp{x_*}}{|x_*|} \right | = 
	\left | \frac{|x_*| x}{|x|}-\frac{|x|x_*}{|x_*|} \right |	=|x-x_*|.
\]
This could be shown expanding the square of the two scalar products, or by remarking that the vectors $\frac{|x_*|x}{|x|}$ and $\frac{|x|x_*}{|x_*|}$ are the images of $x$ and $x_*$ by the symmetry with respect to the bisector of the angle between $x$ and $x_*$.
Dividing by $|x||x_*|$ we get 
\[
	\left | \frac{\pp{x}}{|x|^2}-\frac{\pp{x_*}}{|x_*|^2} \right | = 
	\left | \frac{ x}{|x|^2}-\frac{x_*}{|x_*|^2} \right |	=\frac{|x-x_*|}{|x_*||x|}
\]
Then, using for all $a,b>0$
\[
	\frac{1}{ab} = \frac{ab}{a^2 b^2} \leq \frac{1}{2} \frac{a^2+b^2}{a^2 b^2} = \frac{1}{2} \left (\frac{1}{a^2}+\frac{1}{b^2} \right ),
\]
we obtain the claimed result.
\end{proof}

The previous lemma allows to control the variation of $J$ in a convenient way. Precisely
\begin{lemma} \label{lem:var_J}
For any $x,v,x_*,v_* \in \RR^2 \setminus \{0\}$, denoting $ \delta := |x-x_\ast| + |v-v_\ast|$, 
\begin{multline*}
\bigl| J(x,v) - J(x_*,v_*) \bigr| \le
\left(\frac{\one_{|v| \le |x| }}{|x|^2} + \frac{\one_{|v_*| \le |x_*| }}{|x_*|^2}  \right) \frac{|x-x_*|}{4 \pi}  \\ +
\frac1{2 \pi |x|} \one_{|v| \le |x| \le |v| +  \delta} +
\frac1{2 \pi |x_*|} \one_{|v_*| \le |x_*| \le |v_*| +  \delta}
\end{multline*}
\end{lemma}

Before the proof of that lemma, we try to explain why that bound allows to get simpler and stronger estimates than in the Vlasov-Poisson case. The first term in the r.h.s is similar to what you can get in the Vlasov-Poisson system, with a major difference : the presence of the indicator function. It has two vantages: 
\begin{itemize}
    \item The singularity is screened when $x$ (or $x_*$) approaches $0$,
    \item The bound vanishes for large $v$, so that the control of the spatial density $\rho_f$ will not be required to control this terms.
\end{itemize}
The presence of that indicator function has however a drawback: the presence of the second and third terms. So an important novelty w.r.t. the VP case is that we will need to control new terms linked to the indicator function. But these terms are not too singular because:
\begin{itemize}
    \item The singularity is lower (integrable near $0$) and still screened,
    \item they vanish outside a gear, which is quite thin when $(x,v)$ and $(x_*,v_*)$ are close together.
\end{itemize}

\begin{proof}
By the definition~\eqref{eq:def_J}
\begin{align*}
   2 \pi  \bigl| J(x,v) - J(x_*,v_*) \bigr| &  = \left | \one_{|v| \leq |x|}\frac{\pp{x}}{|x|^2}- \one_{|v_*| \leq |x_*|}\frac{\pp{x_*}}{|x_*|^2} \right |  \\
   & = \one_{|v| \leq |x|}  \one_{|v_*| \leq |x_*|} \left | \frac{\pp{x}}{|x|^2}-\frac{\pp{x_*}}{|x_*|^2} \right |  +
    \frac{\one_{|v| \leq |x|} \one_{|v_*| \geq |x_*|}}{|x|} +
    \frac{\one_{|v| \geq |x|} \one_{|v_*| \leq |x_*|}}{|x_*|} \\
    & = I_1 + I_2 + I_2^\ast.
\end{align*}
The term $I_1$ is bounded by the help of Lemma~\ref{lem:ineq_R2}.
\begin{align*}
I_1 & \leq \one_{|v| \leq |x|}  \one_{|v_*| \leq |x_*|} \left(\frac{1}{|x|^2}+\frac{1}{|x_*|^2} \right ) \frac{|x-x_*|}2  \\ 
& \leq  \left(\frac{\one_{|v| \leq |x|}}{|x|^2}+\frac{\one_{|v_*| \leq |x_*|}}{|x_*|^2} \right ) \frac{|x-x_*|}2.
\end{align*}

The $I_2$ term is bounded using a geometric argument, called ```rope argument'' in~\cite{HAU2012}. In fact, it is non zero only if $|x| - |v|$ has a positive sign and $|x_*|-|v_*|$ a negative one. This happens only if $\delta  = |x-x_*| + |v-v_*|$ is not too small. Precisely, in that case
\[
|x|- |v| \le   (|x|-|v|) - ( |x_*| - |v_*| ) = (|x| -|x_*|) - (|v| - |v_*|) 
\le  |x-x_*| + |v-v_*| = \delta 
\]
by the triangular inequality. It implies that
\[
I_2 \le \frac1{|x|} \one_{|v| \le |x| \le |v| +  \delta}.
\]
The term $I_2^\ast$ is bounded in the same way that $I_2$, exchanging the role of $(x,v)$ and $(x_*,v_*)$ :
\[
I_2^\ast \le \frac1{|x_*|} \one_{|v_*| \le |x_*| \le |v_*| +  \delta}.
\]
Summing up the above bounds on $I_1,I_2,I_2^\ast$, we obtain the claimed inequality.
\end{proof}

We now state two very useful lemmas, about the possibility to bound some specific integral with the help of the $L^\infty_\gamma$ norm. They will  allow us to bound integrals of the terms appearing in the r.h.s. of the bound of Lemma~\ref{lem:var_J}: Lemma~\ref{lem:int_bound_1} will be useful for the first term, and Lemma~\ref{lem:int_bound_2}  for the second and third terms.

\begin{lemma}
\label{lem:int_bound_1}
Let $g$ be a function in $L^1 \cap L^\infty(\RR^2)$, $x,v \in \RR^2$. Then,
\[
\int_{|w-v|<|y-x|} \frac{g(\dd y,\dd w)}{2 \pi |y-x|^2} \le  \bigl( \|g\|_1 \|g \|_\infty \bigr)^{\frac12} \le  \kappa_\gamma  \|g\|_\gamma,
\]
the last bound being valid for $\gamma >2$, and if $\|g\|_\gamma < \infty$. $\kappa_\gamma$ stands for the for the constant introduced in~\eqref{eq:prop_norm_gamma}.
\end{lemma}

\begin{proof}
Picking up an $R>0$ and using Fubini Theorem, we can simply bound
\begin{align*}
	\int_{|v-w| < |x-y| <R }  \frac{g(\dd y, \dd w)}{|x-y|^2} &\leq 
	\|g\|_\infty \int_{|x-y| <R}  \frac{\dd y}{|x-y|^2} \underbrace{\int_{|x-y|>|v-w|} \dd w}_{=\pi|x-y|^2} \\
	&= \pi \|g\|_\infty \int_{|x-y| <R} \dd y = \pi^2 \|g\|_\infty R^2.
\end{align*}
And also easily
\[
\int_{|x-y| \ge R }  \frac{g(\dd y, \dd w)}{|x-y|^2} \le \frac{\|g\|_1}{R^2}.
\]
Summing these two inequalities, we obtain
\[
\int_{|v-w| < |x-y|  }  \frac{g(\dd y, \dd w)}{|x-y|^2} \le \pi^2 \|g\|_\infty R^2 +\frac{\|g\|_1}{R^2},
\]
an inequality valid for any $R>0$. It turns out that the optimal choice for $R$ is $R_m^2 = \sqrt{\frac{\|g\|_1}{\pi^2 \|g\|_\infty}}$, in which case the bound is equal to $2 \pi \bigl( \|g\|_1 \|g \|_\infty \bigr)^{\frac12}$. A quantity smaller than $ 2 \pi \kappa_\gamma  \|g\|_\gamma$ thanks to~\eqref{eq:prop_norm_gamma}. Dividing by $2\pi$, the conclusion follows.
\end{proof}


\begin{lemma}
\label{lem:int_bound_2}
Let $\gamma >2$, $g$ be a probability distribution in $L^\infty_\gamma$, $x,v \in \RR^2$ and  $\delta >0$. Then, 
\[
\int_{ |v-w| \le |x-y|  < |v-w| + \delta} \frac{g(\dd y,\dd w)}{2 \pi |y-x|} \le \kappa_\gamma \|g\|_\gamma \delta.
\]
\end{lemma}

\begin{proof}
\begin{align*}
\int_{|v-w| \le |x-y|  < |v-w| + \delta} \frac{g(\dd y,\dd w)}{|y-x|}  
& \leq 
\|g\|_{\gamma} \int_{\RR^2} \frac{\dd w}{(1+|w|)^{\gamma}}  \int_{|v-w| < |y-x|<|v-w|+\delta} \frac{\dd y}{|x-y|} \\
&  = 
\|g\|_{\gamma} \int_{\RR^2} \frac{\dd w}{(1+|w|)^{\gamma}}  \int_{|v-w|}^{|v-w|+\delta} 2 \pi \dd r \\
& \leq  2 \pi \delta \|g\|_{\gamma} \int_{\RR^2} \frac{\dd w}{(1+|w|)^{\gamma}} = 2 \pi \kappa_\gamma \|g\|_{\gamma} \delta.
\end{align*}
The integral in $y$ on the first line is in fact computed after a radial change of variable. Dividing by $2 \pi$, the conclusion follows.
\end{proof}

\subsection{A priori bound and Lipschitz estimates on the advection fields}

Our first important estimate is an infinite bound on the advection fields. In the 3D Vlasov-Poisson system, if we denote by $E_f$ the electric field induced by the distribution $f$, we have an estimate of the form
for any $m>3$
\[
	\|E_f \|_{\infty} \leq C_m \|\rho_f\|_m \le  C_m' \left(\|\rho_f\|_1+\|\rho_f\|_{\infty} \right),
\]
where $\rho = \int f \, \dd v$ is the density. 
But propagating uniform bound on the density is not so easy. This was however performed by~\citeauthor{LIO1991} in~\autocite{LIO1991} and by \citeauthor{Pfaff} in~\autocite{Pfaff} with very different technics.

In our case, it is actually possible to get a simpler and better estimate on $\|\vit[f]\|_{\infty}$ and $\|\acc[f]\|_{\infty}$ by taking advantage of the indicator function present in the definition of the fields in~\eqref{vit-expr} and~\eqref{acc-expr}. That bound is much easier to satisfy because of the propagation of the $L^p$ norms.  This is what the following proposition is about. This is an important estimate in our work. It allows to control the advection fields with bound on $\|f\|_1$ and $\|f\|_\infty$ only, which are quite simple to propagate by Proposition~\ref{prop-cons-norm} below.

\begin{prop}\label{bounded-fields}
For any $f \in L^1(\RR^4)\cap L^{\infty}(\RR^4)$, the following estimates holds :
\[
	\left\|\vit[f] \right\|_{\infty}, \left\|\acc[f] \right\|_{\infty}  \leq c \,
	\|f\|_{\infty}^{\frac 14} \|f\|_1^{\frac 34} \leq c\, \kappa_\gamma^{\frac 32} \|f\|_\gamma, 
	\quad \text {with } \quad c:= \frac{2^{\frac 54}}{3 \sqrt \pi}.
\]
\end{prop}
the last estimate being valid for $\gamma >2$, and $\|f\|_\gamma < \infty$.
\begin{proof}
    We only write the proof for the bound of $\vit[f]$. The bound for $\acc[f]$ can be obtained by a permutation of the variables $x$ and $v$ below. By definition~\eqref{eq:def_J} and taking into account the indicator function in it, we write :
	\begin{align*}
		\left | \vit[f](x,v) \right | &\leq \frac{1}{2 \pi} \int_{\RR^2} \frac{\dd y}{|x-y|} \int_{|v-w|<|x-y|} f(y,w) \, \dd w \\
		&= \frac{1}{2 \pi} \left [ \int_{|x-y| < R} + \int_{|x-y| \geq R} \right ] \frac{\dd y}{|x-y|} \int_{|v-w|<|x-y|} f(y,w) \, \dd w \\
		&\leq \frac{1}{2\pi} \|f\|_{\infty} \int_{|x-y|<R} \frac{1}{|x-y|} \underbrace{\left ( \int_{|v-w|<|x-y|} \, \dd w \right )}_{= \pi |x-y|^2} \, \dd y + \frac{1}{2\pi R } \|f\|_1 \\
		&= \frac{1}{2} \|f\|_{\infty} \int_{|x-y|<R} |x-y| \, \dd y + \frac{1}{2\pi R } \|f\|_1 = \frac{\pi R^3}{3} \|f\|_{\infty} +  \frac{1}{2\pi R} \|f\|_1.
	\end{align*}
	It remains to perform an optimization on $R$. The bound in the r.h.s. is of the form $f(R) = a R^3 + \frac b R$, with $a= \frac \pi 3 \|f\|_\infty, \; b= \frac1 {2 \pi} \|f\|_1$. A short analysis shows that $f$ is minimal when $R=R_m = \left( \frac b {3a} \right)^{\frac 14}$, with a minimal value $f_m = \frac 4 {3^{\frac 34}} b^{\frac 34} a^{\frac14}$. Replacing $a$ and $b$ by their values gives the expected result.
	The second bound is a consequence of~\eqref{eq:prop_norm_gamma}.
\end{proof}

For the 2D and 3D Vlasov-Poisson systems, it is known that the electric field is log-Lipschitz continuous when $\rho_f \in L^{\infty}$. The proof is classical and can be found in \autocite[lemma~8.1]{MAJ2002} in dimension three or \autocite[lemma~3.1]{MAR1994} in dimension two. This enables to define classical characteristics thanks to the Osgood's criterion. In the present case, the fields are actually Lipschitz continuous, when $f \in L^\infty_\gamma$ for $\gamma>2$, as the following proposition shows.
This shows again that the situation is simpler here.

\begin{prop}\label{prop-fields-lipschitz}
	Let $f$ such that $\|f\|_{\gamma}<+\infty$ for some $\gamma>2$.
	Then the fields $\vit[f]$ and $\acc[f]$ defined respectively by~\eqref{vit-expr} and~\eqref{acc-expr}, are Lipschitz continuous in $(x,v)$, with a Lipschitz constant depending only on $\|f\|_{\gamma}$ : 
    \[
    \left\|  \vit[f] \right\|_{\text{Lip}} :=  \sup_{x,v,x_*,v_*} \frac{\left | \vit[f](x,v) - \vit[f](x_*,v_*) \right |}{|x-x_*| + |v-v_*|} \leq  3 \kappa_\gamma \| f \|_\gamma.
    \]
    and  similarly $\left\| \acc[f] \right\|_{\text{Lip}} \leq 3 \kappa_\gamma \| f \|_\gamma$.
\end{prop}

\begin{proof}
Form the expression~\eqref{vit-expr}, we obtain the bound
\[
\left | \vit[f](x,v) - \vit[f](x_*,v_*) \right | \leq \frac{1}{2\pi} \int f(\dd y, \dd w) 
\left | J(x-y,v-w) -  J(x_*-y,v_*-w)  \right |.
\]
Using  the bound of Lemma~\ref{lem:var_J} inside the integral with the notation $\delta= |x-x_*|+|v-v_*|$ leads to the bound
\begin{align*}
| \vit[f](x,v)  & - \vit[f](x_*,v_*)  |  \leq (J_1 + J_1^*)|x-x_*| +  J_2 + J_2^*,  \quad  \text{where}  \\
J_1 &=   \int_{|v-w| \le |x-y| }  \frac{ f(\dd y, \dd w) }{4 \pi |x-y|^2}, \qquad \qquad \quad \;\;
J_1^* =   \int_{|v_*-w| \le |x_*-y| }  \frac{ f(\dd y, \dd w) }{4 \pi |x_*-y|^2}, \\
J_2 & =  \int_{|v-w| \le |x-y| \le |v-w| + \delta}  \frac{f(\dd y, \dd w) }{2 \pi |x-y|}, \qquad 
J_2^* =  \int_{|v_*-w| \le |x_*-y| \le |v_*-w| + \delta}  \frac{f(\dd y, \dd w) }{ 2 \pi |x_*-y|}.
\end{align*}
The integrals $J_1$ and $J_1^*$ may be bounded with the help of Lemma~\ref{lem:int_bound_1}:
\[
J_1 + J_1^* \le \kappa_\gamma \| f \|_\gamma.
\]
The integrals $J_2$ and $J_2^*$ may be bounded with the help of Lemma~\ref{lem:int_bound_2}:
\[
J_2 + J_2^* \le 2   \kappa_\gamma \| f \|_\gamma \delta  = 2 \kappa_\gamma \| f \|_\gamma \bigl( |x-x_*| + |v-v_*| \bigr).
\]
Summing this two bounds leads to 
\[
\left | \vit[f](x,v) - \vit[f](x_*,v_*) \right | \leq 3 \kappa_\gamma \| f \|_\gamma \bigl( |x-x_*| + |v-v_*| \bigr). 
\]
\end{proof}

\subsection{Transport and continuity equation, pushforward and preservation of norms.} \label{subsec:tr_eq}

We recall here the standard resolution of the transport and continuity equation, applied to our particular case. In the meantime, we state precisely some classical definition that will play an important role later. 

The advection fields defined by~\eqref{relation-vit-acc-pot} are divergence free. A case which allows to rewrite the transport equation~\eqref{limit-equation} in a conservative form, also called a \emph{continuity equation}
\begin{equation}\label{eq:continuity}
\partial_t f_t + \divg_x ( \vit[f_t] f_t ) + \divg_v ( \acc[f_t] f_t ) = 0.
\end{equation}
When the field $\vit[f_t]$ and $\acc[f_t]$ are Lipschitz, the continuity equation is usually solved with the help of the characteristics, that is the trajectories associated to ODE driven by the vector field $(\vit[f_t],\acc[f_t])$:
\begin{equation} \label{eq:characteristics}
X_t(z)  = x + \int_0^t \vit[f_s](X_s(z), V_s(z))\dd s, 
\qquad 
V_t(z)  = v + \int_0^t \vit[f_s](V_s(z), V_s(z))\dd s.
\end{equation}
where $z=(x,v)$. the collection of mapping $Z_t=(X_t,V_t)$ for $t \ge 0$ is called the \emph{flow} associated to the vector field.

As a consequence of the Cauchy-Lipschitz theorem\footnote{which applies here as the vector field is assumed to be Lipschitz in position-velocity, with a constant that is locally integrable in time}, the  mapping $Z_t$ is invertible, with in fact a continuous inverse.
At a given time $t$, $Z_t^{-1}$ is in fact the flow associated to the the time reversed vector field $(s,x,v) \to \left( - \vit[f_{t-s}], - \acc[f_{t-s}]\right)$. It satisfies $Z_t^{-1}=Z^{rev}_{t,t}$ with 
\begin{equation} \label{eq:char_reverse}
Z^{rev}_{t,s}(z) = z - \int_0^s \left( \vit[f_{t-u}],\acc[f_{t-u}]\right)(Z^{rev}_{t,u}(z)) \, \dd u,
\end{equation}
Such properties could be found in standard book about ODE, like~\autocite{Hirsch1974}.

In the case where the vector field is divergence free, a property that is always true in our setting, then the flow is volume preserving. This means precisely that for any continuous test function $\varphi : \RR^4 \to \RR$ with bounded support, and any time $t \ge 0$
\begin{equation}\label{eq:measure_preserving}
\int_{\RR^4} \varphi(Z_t(z)) \dd z = \int_{\RR^4} \varphi(z) \dd z.
\end{equation}
To see why this property is related to the preservation of volume, use indicatrix function of set $A$ (with neglectable boundary) instead of continuous function in the above equality, and get $\mathrm{Vol}(Z_t^{-1}(A)) = \mathrm{Vol}(A)$.

For the sake of completeness, we redefine the notion of pushforward.
\begin{definition} \label{def:pushforward}
Let $f \in \proba(\RR^4)$ and $T : \RR^4 \to \RR^4$ a measurable map. Then the pushforward of $f$ by the map $T$ is defined as the probability measure denoted $\pushfwd{T}f$ such that for any measurable set $A \subset \RR^4$, $\pushfwd{T}f(A)  = f(T^{-1}(A))$. Equivalently, for any continuous test function $\varphi :\RR^4 \to \RR$, the following equality holds
\begin{equation} \label{eq:pushforward}
\int_{\RR^4} \varphi(z) \left[\pushfwd{T}f \right](\dd z) = \int_{\RR^4} \varphi(T(z)) f(\dd z). 
\end{equation}
When the map $T$ is volume preserving and has a measurable inverse $T^{-1}$, and $f$ has a density (still denoted by $f$ here), then 
$\pushfwd{T}f$ also has a density and
\[
(\pushfwd{T}f)(z)  = f(T^{-1}(z)).
\]
\end{definition}
The last point is a consequence of the measure preserving property~\eqref{eq:measure_preserving} applied to the equality~\eqref{eq:pushforward} defining the pushforward. 

The solution of the continuity equation writes as the pushforward of the initial condition by the flow. It is also classical that the divergence free property implies the preservation of $L^p$ norms of solutions  along time. We prove this rigorously in the next proposition.

\begin{prop}[Conservation of $L^p$ norms]
\label{prop-cons-norm}
	Let $f \in L^1_{\mathrm loc}(\RR^+,L^\infty_\gamma(\RR^4))$ for some $\gamma >2$ be a solution to~\eqref{limit-equation} in the sense of Definition~\ref{def:sol}. Then for any positive time $t$, $f_t$ is given by the pushforward of $f_0$ by the map $Z_t$ defined in~\eqref{eq:characteristics}
	\[
	f_t := \pushfwd{Z_t}f_0.
	\]
	Moreover, if $f_0 \in L^1 \cap L^\infty(\RR^4)$, then for any $p \in [1,+\infty]$ and any time $t>0$,
	\[
	 \|f_t\|_{L^p_{x,v}} = \|f_0\|_{L^p_{x,v}},
	\]
	
\end{prop}

\begin{proof}
According to Proposition~\ref{prop-fields-lipschitz} the advection fields $\vit[f_t]$ and $\acc[f_t]$ are Lipschitz for almost every time $t\ge 0$.
In that case, the mappings $Z_t$ for $t \ge 0$ are well defined thanks to the Cauchy-Lipschitz theorem, and inversible. Since the field $(\vit[f_t],\acc[f_t])$ is divergence free, they also are measure preserving. So we can write according to Definition~\ref{def:pushforward}
\[
f_t(z) =  \pushfwd{Z_t}f_0 (z)  = f_0(Z_t^{-1}(z)).
\]

Since $Z_t$ is measure preserving, an application of~\eqref{eq:measure_preserving} leads for any $p \in [1,\infty)$
\[
\|f_t\|_{L^p}^p =  \int_{\RR^4} |f_0(Z_t^{-1}(z))|^p \, \dd z
=  \int_{\RR^4} |f_0(z)|^p \, \dd z
= \|f_0\|_{L^p}^p.
\]
The case $p=\infty$ follows as a limit as $p \to \infty$.

\end{proof}

Next, the boundedness of the fields enables to show the propagation of the $\|f_t\|_\gamma$ norms, as shows the next proposition.
\begin{prop}\label{bound-decreasing-f}
	Let  $\gamma>2$. For any solution $f \in L^1_{\mathrm loc}(\RR^+,L^\infty_\gamma)$ to~\eqref{limit-equation} with $f_0$ as initial datum satisfying $\|f_0\|_{\gamma}<+\infty$,
    \begin{equation}\label{decay-property-f_t}
	\|f_t\|_{\gamma} \leq \left( 1+ 2 c \, \|f_0 \|_\infty^{\frac14} \|f_0\|_1^{\frac34} t \right)^{2\gamma}\|f_0\|_{\gamma}, \quad t\geq 0,
    \end{equation}
    where $c$ is the constant defined in Proposition~\ref{bounded-fields}.
\end{prop}
\begin{proof}
	We follow the proof done in~\autocite{UKA1978} for the Vlasov-Poisson system. The present case is however more favorable because the decay exponent $\gamma$ is the same for $f_0$ and $f_t$, whereas it is divided by $2$ in the Vlasov-Poisson case.
	
	As $\gamma>2$, $f_0 \in (L^1 \cap L^{\infty})(\RR^4)$ by~\eqref{eq:prop_norm_gamma}. So does $f_t$ by Proposition~\ref{prop-cons-norm} about propagation of $L^p$ norms, which applies here because $f$ is assumed to be smooth enough. And $f_t(z) = f_0(Z_t^{-1}(z))$
	
	Then, $\vit[f_t]$ and $\acc[f_t]$ are bounded according to Proposition~\ref{bounded-fields}. And with the constant $c$ defined in there, 	the integration of the reverse characteristics~\eqref{eq:char_reverse} gives
	\begin{align}
		Z_t^{-1}(z) - z  &=  - \int_0^t \left(\vit[f_{t-s}],\acc[f_{t-s}]\right)(Z^{rev}_{t,s}) \, \dd s,  	\qquad \nonumber \\
		\left| Z_t^{-1}(z)  - z \right| &  \le 2 \, c  \int_0^t \|f_s \|_\infty^{\frac14} \|f_s\|_1^{\frac34} \dd s. \label{eq:control_traj}
	\end{align}
	Thanks to Proposition~\ref{prop-cons-norm} that applies here in view of the hypothesis, we simplify in
\[
	\left| Z_t^{-1}(z)  - z \right|  \le 2 c \,  t \, \|f_0 \|_\infty^{\frac14} \|f_0\|_1^{\frac34}.
\]
	Then, using the following inequality valid for all $x,x' \in \RR^d$ 
	\begin{equation}\label{ineq-x-x'}
		\frac{1+|x|}{1+|x'|} \leq 1+|x-x'|,
	\end{equation}
	it comes :
	\begin{align*}
	(1+|x|)^{\gamma} (1+|v|)^{\gamma}	f_t(z) &=  (1+|x|)^{\gamma} (1+|v|)^{\gamma} f_0(Z_t^{-1}(z)) \\
	& \leq  \|f_0\|_{\gamma} \frac{ (1+|x|)^{\gamma} (1+|v|)^{\gamma}}{\left(1+\left|(Z_t^{-1})_x(z)\right|\right)^{\gamma} \left(1+\left|(Z_t^{-1})_v\right|\right)^{\gamma}} \\
		&\le   \|f_0\|_{\gamma} \left(1+\left | (Z_t^{-1})_x(z) - x \right| \right)^{\gamma} \left( 1+\left|  (Z_t^{-1})_v(z) - v \right| \right)^{\gamma} \\
\|f_t\|_\gamma		&\leq  \|f_0\|_{\gamma} \left(1+ 2 c \,  t \, \|f_0 \|_\infty^{\frac14} \|f_0\|_1^{\frac34}\right)^{2\gamma},
	\end{align*}
thanks to~\eqref{eq:control_traj} and a supremum on $x$ and $v$ for the last line.
\end{proof}

\subsection{Wasserstein distance}

One can equip $\proba_1(\RR^4)$ (defined in Subsection~\ref{subsec:not}) with the Wasserstein distance of order $1$ defined by
\[
	\wass_1(\mu,\nu) = \inf_{\pi \in \Pi(\mu,\nu)} \int_{\RR^4} |x-y| \, \dd \pi(x,y), \quad \mu,\nu \in\proba_1(\RR^4),
\]
where $\Pi(\mu,\nu)$ is the set of all \emph{transference planes}, i.e. probability measures on $\RR^4 \times \RR^4$ whose marginals are $\mu$ and $\nu$.
If $\mu$ and $\nu$ are absolutely continuous with respect to the Lebesgue measure, then one also has
\[
\wass_1(\mu,\nu) = \inf_{\pushfwd{\transpt} \mu = \nu} \int_{\RR^4} |x- \transpt(x)| \, \dd \mu(x),
\]
where the infimum runs over all \emph{transport maps}, i.e. measurable maps $T:\RR^4 \rightarrow \RR^4$ that push forward $\mu$ onto $\nu$, that is to say $\pushfwd{\transpt}\mu = \nu$ according to Definition~\ref{def:pushforward}. 
As we will only deal with absolutely continuous functions, the latter formulation will be adopted. 
This is not mandatory, its main vantage is to simplify the notation.
For much more information about optimal transport and Wasserstein distances, one refers to~\autocite{VIL2003,ambrosio}.

\section{Wasserstein stability for the limit equation}\label{stability-wasserstein}

This section is dedicated to the proof of theorem~\ref{theorem-stability}.  
We construct an appropriate coupling at any time $t$ between solutions of Equation~\eqref{limit-equation}. Then, we carefully control the evolution of the transport cost along time. That is we  estimate the difference in the advection fields created by the two solutions, with the help of the technical lemmas of the previous section. After that, a Grönwall argument allows to conclude.


\subsection{The appropriate coupling}

Let $f,g  \in L^1_{\mathrm{loc}}(\RR^+,L^\infty_\gamma(\RR^4))$ be two solutions to~\eqref{limit-equation} in the sense of Definition~\ref{def:sol}.
The goal is to find a control of the distance between $f_t$ and $g_t$ for any time $t\geq 0$, depending on their distance at $t=0$. The Wasserstein metric will be used to quantify the gap between two solutions, a tool well adapted for this kind of problems.
Let $\transpt_0 : \RR^4 \to \RR^4$ be an optimal transport map transporting $f_0$ onto $g_0$, 
which always exists because $f_0$ has a density (See~\cite{VIL2003}). Namely
\begin{equation}\label{T_0-stability}
	\wass_1(f_0,g_0) = \int_{\RR^4} \left |(x,v)-\transpt_0(x,v) \right | \, f_0(\dd x, \dd v)
	= \int_{\RR^4} \left |z-\transpt_0(z) \right | \, f_0(\dd z),
\end{equation}
where $z=(x,v)$. Let $Z_t^f = (X_t^f,V_t^f)$ and $Z_t^g = (X_t^g,V_t^g)$ be the characteristics at time $t$ associated to $f_t$ and $g_t$ respectively, having the value $(x,v)$ for $t=0$ as in~\eqref{eq:characteristics}. They exist because the fields are Lipschitz continuous according to proposition~\ref{prop-fields-lipschitz}. 

According to Proposition~\ref{prop-cons-norm}, $Z_t^f$ transports $f_0$ onto $f_t$ and $Z_t^g$ transports $g_0$ onto $g_t$, i.e.
\[
	f_t = \pushfwd{Z_t^f} f_0, \qquad 
	g_t = \pushfwd{Z_t^g} g_0.
\]
Since $Z_t^f$ is smooth and invertible, one can thus build a transport map $\transpt_t$ transporting $f_t$ onto $g_t$ by (See Figure~\ref{fig:Tmap})
\begin{equation} \label{def:transport}
	\transpt_t = Z_t^g \circ \transpt_0 \circ (Z_t^f)^{-1},
\end{equation}
which is illustrated by figure~\ref{transport-limit-equation-two-solutions}. That family of transport map will allow to control the transport cost $W(f_t,g_t)$.
\begin{figure} \label{fig:Tmap}
	\centering
	\begin{tikzpicture}
	\node at (0,0) {$f_t$};
	\node at (2,0) {$g_t$};
	\node at (0,2) {$f_0$};
	\node at (2,2) {$g_0$};
	\draw[->] (0.3,2) -- (1.7,2);
	\draw[->] (0.3,0) -- (1.7,0);
	\draw[->] (0,0.3) -- (0,1.7);
	\draw[->] (2,1.7) -- (2,0.3);
	\node[above] at (1,2) {$\transpt_0$};
	\node[below] at (1,0) {$\transpt_t$}; 
	\node[left] at (0,1) {$(Z_t^f)^{-1}$};
	\node[right] at (2,1) {$Z_t^g$};
	\end{tikzpicture}
	\caption{Transport between two solutions to~\eqref{limit-equation}}
	\label{transport-limit-equation-two-solutions}
\end{figure}

\subsection{Estimate on the transport cost}

By definition of $\wass_1$, and denoting $\transpt_t=(\transpt_t^x,\transpt_t^v)$, we have
\begin{equation} \label{deft-Q-uniqueness}
\begin{aligned}
	& \hspace{2cm} \wass_1(f_t,g_t) \leq Q(t)  \quad \text{where} \\ 
	Q(t) & := \int_{\RR^4} \left |z-\transpt_t(z) \right | \, f_t(\dd z)
	= \int_{\RR^4} \left | \left ( x-\transpt_t^x(x,v), v-\transpt_t^v(x,v) \right ) \right | \, f_t(\dd x, \dd v) \\
	&= \int_{\RR^4} \left | \left ( X_t^f(x,v)-X_t^g(\transpt_0(x,v)), V_t^f(x,v)-V_t^g(\transpt_0(x,v)) \right ) \right | \, f_0(\dd x, \dd v) \\
	& \le \int_{\RR^4} \left |  X_t^f(z)-X_t^g(\transpt_0(z)) \right | \, f_0(\dd z)
	+ \int_{\RR^4} \left | V_t^f(z)-V_t^g(\transpt_0(z))  \right| \, f_0(\dd z).
\end{aligned}
\end{equation}
In view of the definition~\eqref{eq:characteristics} of the characteristics, we write
\begin{align*}
\left |  X_t^f(z)-X_t^g(\transpt_0(z)) \right | &=
\left |  \int_0^t \left(\vit[f_s](Z_s^f(z))- \vit[g_s](Z_s^g(\transpt_0(z)))\right) \,\dd s \right | \\
& = \left |  \int_0^t \left(\vit[f_s](Z_s^f(z))- \vit[g_s](\transpt_t(Z_s^f(z)))\right) \,\dd s \right |,
\end{align*}
where we used the definition~\ref{def:transport} of the transport $\transpt_t$. A similar equality holds for the acceleration term $\acc$. Integrating~\eqref{deft-Q-uniqueness} in $z$ w.r.t. $f_0$ and using Fubini, we obtain
\[
Q(t)  \le \int_0^t I(s) \, \dd s +  \int_0^t J(s) \, \dd s
\]

\begin{align*}
\text{with } \quad I(s)  & :=  \int_{\RR^4} \left |  \vit[f_s](Z_s^f(z)) - \vit[g_s](\transpt_s(Z_s^f(z))) \right | \, f_0(\dd z) \\
	& =  \int_{\RR^4} \big | \vit[f_s](z) - \vit[g_s](\transpt_s(z)) \big | \, f_s(\dd z)  \\
 J(s) &  := \int_{\RR^4} \big | \acc[f_s](z) - \acc[g_s](\transpt_t(z)) \big | \, f_s(\dd z) 
\end{align*}
To get the second line, we used the fact that $f_t = \pushfwd{Z^f_t} f_0$ and the change of variable property~\eqref{eq:pushforward}.
Both $I$ and $J$ will be bounded with the help of the following proposition, which is in fact a bit more general that what is strictly requested here. It allows the use of different densities for the calculation of the fields, a hypothesis that will be useful for the proof of Theorem~\ref{theorem-existence}. Here we will apply it only in a particular case.

\begin{prop} \label{prop:key_estim}
Let $\gamma > 2$. Assume that $f,g, \tilde f, \tilde g \in \mathcal P(\RR^4) \cap L^\infty_\gamma(\RR^2)$. Assume also that $\transpt$ (resp. $\ttranspt$) is a mapping from $\RR^4$ onto itself that transport $f$ onto $g$: $\pushfwd{\transpt} f=g$ (resp. $\tilde f$ onto $\tilde g$: $\pushfwd{\ttranspt} \tilde f= \tilde g$). The following bound holds
\begin{align*}
\int f(\dd z) \left| \vit[\tilde f](z) - \vit[ \tilde g](\transpt(z)) \right|  & \le \frac{3 \kappa_\gamma}{\sqrt 2} \left( \bigl(\| \tilde f \|_\gamma + \| \tilde g \|_\gamma\bigr) Q + \bigl(\|  f \|_\gamma + \|  g \|_\gamma \bigr) \tilde Q \right)  \\
 \text{where } \quad 
Q & : = \int_{\RR^4} f(\dd z) \left|z - \transpt(z) \right| \quad \text{and }\quad 
\tilde Q : = \int_{\RR^4} \tilde f(\dd z) \left|z - \ttranspt(z) \right|.
\end{align*}
\end{prop}

That proposition is the main calculation of our work. It allows to obtain the Grönwall estimate that is necessary to perform the proof of the stability result of Theorem~\ref{theorem-stability} and of the existence result of Theorem~\ref{theorem-existence} (that relies in fact on a stability estimate between two solutions of slightly mollified equation). We postpone its proof to the next section, and now apply it in order to estimate $Q(t)$ and eventually  conclude the proof.

\subsection{Conclusion of the proof of Theorem~\ref{theorem-stability}}
We recall that 
\[
I(t) = \int_{\RR^4} \big | \vit[f_t](z) - \vit[g_t](\transpt_t(z)) \big | \, f_t(\dd z)
\]
To bound it, we apply Proposition~\ref{prop:key_estim} with $f= \tilde f = f_t$, $g= \tilde g = g_t$, and $\transpt = \ttranspt = \transpt_t$. We get
\[
I(t)  \leq 3 \sqrt2 \, \kappa_\gamma \bigl(\| f_t \|_\gamma + \|g_t \|_\gamma\bigr) Q(t).
\]
In view of the symmetry between $\vit[f]$ and $\acc[f]$, the same bound for $\acc[f]$ is obtained exchanging the role of $x$ and $v$. It leads to the Gr\"onwall estimate
\[
Q(t)  \le  \int_0^t \left[I(u) + J(u) \right] \, \dd u   \le 6 \sqrt 2 \, \kappa'_\gamma \int_0^t \bigl(\| f_u \|_\gamma + \|g_u \|_\gamma\bigr) Q(u) \, \dd u, \qquad  
\kappa_\gamma' = 6 \sqrt 2 \, \kappa_\gamma
.
\]
An application of the usual Gr\"onwall lemma leads to
\[
Q(t)  \le Q(0) e^{\kappa_\gamma' \int_0^t (\|f_s\|_{\gamma}+\|g_s\|_{\gamma}) \dd s}
\]
Recalling that (cf.~\eqref{deft-Q-uniqueness}  and~\eqref{T_0-stability})
$ \wass_1(f_t,g_t) \leq Q(t)$ and $Q(0)=\wass_1(f_0,g_0)$,
the claimed stability estimate follows  for all $t \ge 0$,
\[
\wass_1(f_t,g_t) \leq \wass_1(f_0,g_0) e^{6 \kappa_\gamma \int_0^t (\|f_s\|_{\gamma}+\|g_s\|_{\gamma}) \dd s}.
\]
This concludes the proof of Theorem~\ref{theorem-stability}.

\subsection{Proof of the Proposition~\ref{prop:key_estim}}

\begin{proof}[Proof of the Proposition~\ref{prop:key_estim}]

\mbox{}

\textsc{Step 1 : Separation of the integral}
We first use the transport relations $g = \pushfwd{\transpt} f$ and $\tilde g = \pushfwd{\tilde \transpt} \tilde f$ and Definition~\ref{def:pushforward} to write the term to bound as a integral involving $f$ and $\tilde f$ only (and not $g$ and $ \tilde g$). Using the definition~\eqref{vit-expr}, we have
\[
\vit[\tilde f](x,v) = \frac{1}{2\pi} \int_{\RR^4} \tilde f(\dd y, \dd w) J(x-y,v-w).
\]
Using the shortcuts
\[
x_* = \transpt^x(x,v), \quad
v_* = \transpt^v(x,v), \quad
y_* = \transpt^x(y,w), \quad
w_* = \transpt^v(y,w)
\]
(which we shall handle carefully as they hide the fact that the variables with a star subscript are in fact functions of the variables without star subscript)  we  write
\begin{align*}
	\vit[\tilde g](\transpt(x,v)) &= \int_{\RR^4} \tilde g(\dd y, \dd w) J(\transpt^x(x,v)-y,\transpt^v(x,v)-w) \\
	&=  \int_{\RR^4} \tilde f (\dd y, \dd w) J(\transpt^x(x,v)-\ttranspt^x(y,w),\transpt^v(x,v)-\ttranspt^v(y,w)) \\
	&= \int_{\RR^4} \tilde f (\dd y, \dd w) J(x_*-y_*,y_*-w_*).
\end{align*}
hence denoting by $I$ the quantity in the l.h.s. of Proposition~\ref{prop:key_estim},
\begin{equation} \label{I-stability}
	I \leq \int_{\RR^8} \hspace{-3mm} f(\dd x, \dd v) \tilde f(\dd y, \dd w) \left| 
	J(x-y,v-w) - J(x_*-y_*,v_*-w_*) \right|. 
\end{equation}

Using the Lemma~\ref{lem:var_J} in the bound~\eqref{I-stability} for $I$, allows to obtain the new bound 
\[
	I\leq I_1 +I_1^*+I_2+I_2^*, \quad \text{with}
\]
\begin{align}
	I_1 &= \iint_{|x-y|>|v-w|} f(\dd x, \dd v) \tilde f(\dd y, \dd w) \frac{|(x-y)-(x_*-y_*)|}{4 \pi |x-y|^2}, \label{I_1}\\
	I_1^* &= \iint_{|x_*-y_*|>|v_*-w_*|} f(\dd x, \dd v) \tilde f(\dd y, \dd w) \frac{|(x-y)-(x_*-y_*)|}{4 \pi |x_*-y_*|^2}, \label{I_3} \\
	I_2 &= \iint_{|v-w| \le |x-y| \le |v-w|  + \bar \delta}  f(\dd x, \dd v) \tilde f(\dd y, \dd w) \frac1{2 \pi |x-y|}, \label{I_5} \\
    I_2^* &= \iint_{|v_*-w_*| \le |x_*-y_*| \le |v_*-w_*|  + \bar \delta}  f(\dd x, \dd v) \tilde f(\dd y, \dd w) \frac1{2 \pi |x_*-y_*|},
\end{align}
where 
\begin{equation} \label{eq:def_delta_bar}
\begin{aligned}
\bar \delta &  = \delta(x,v) + \delta(y,w)
 \le 2\bigl( \delta(x,v) \vee \delta(y,w)  \bigr), \\
 \text{with } \quad \delta(x,v) & := |x-x_*| + |v-v_*| \quad \text{ and } \quad 
 \delta(y,w) = |y-y_*|+|w-w_*|.
\end{aligned}
\end{equation}
with the notation $a\vee b = \max(a,b)$ for any real $a,b$.  We recall that $\delta$ is a function of $(x,v)$ only because $(x_*,v_*)$ depends only on $(x,v)$.

\medskip
\textsc{Step 2 : Estimates on $I_1$ and $I_1^*$}

We estimate $I_1$ with the help of Fubini Theorem and Lemma~\ref{lem:int_bound_1}. The strategy for $I_1^*$ follows the same line with an additional use of the pushforwards $g = \pushfwd{\transpt} f$ and $\tilde g = \pushfwd{\tilde \transpt} \tilde f$.

First we use the triangular inequality to split $|(x-y)-(x_*-y_*)|$ in $|x-x_*| + |y-y_*|$. This is interesting because $x-x_*$ depends only on $(x,v)$ (and $y-y_*$ only on $(y,w)$), a fact that allows to use Fubini theorem.
\begin{multline*}
4 \pi I_1  \le  
\int_{\RR^2} f(\dd x, \dd v) |x-x_*| \int_{|x-y|>|v-w|}   \frac{ \tilde f(\dd y, \dd w)}{|x-y|^2} \\ +
\int_{\RR^2} \tilde f(\dd y, \dd w) |y-y_*| \int_{|x-y|>|v-w|}  \frac{f(\dd x, \dd v)}{|x-y|^2}
\end{multline*}
The two (inside) integrals are bounded using Lemma~\ref{lem:int_bound_1}. It leads to 
\begin{align*}
I_1 & \leq 
\frac{\kappa_\gamma}2 \left( \| \tilde f\|_\gamma \int_{\RR^4} f(\dd x, \dd v) |x-x_*|  +
 \| f\|_\gamma \int_{\RR^4} \tilde f(\dd y, \dd w) |y-y_*| \right)  \\
& \leq \frac{\kappa_\gamma}2 \left( \| \tilde f\|_\gamma Q  + \| f\|_\gamma \tilde Q \right).
\end{align*}

We begin with the same strategy for $I_1^*$:
\begin{multline*}
4 \pi I_1^*  \le  
\int_{\RR^2} f(\dd x, \dd v) |x-x_*| \int_{|x_*-y_*|>|v_*-w_*|}   \frac{\tilde f(\dd y, \dd w)}{|x_*-y_*|^2} \\ +
\int_{\RR^2} \tilde f(\dd y, \dd w) |y-y_*| \int_{|x_*-y_*|>|v_*-w_*|}  \frac{f(\dd x, \dd v)}{|x_*-y_*|^2}. 
\end{multline*}
Then, we could perform the change of variable $(y,w) \to (y_*,w_*) = \ttranspt(y,w)$ in the inside  integral of the first term, and $(x,v) \to (x_*,v_*) = \transpt(x,v)$ in the second term. Integrals w.r.t. the $g$ and $\tilde g$ distributions appear according to the rule~\eqref{eq:pushforward}
\begin{multline*}
4 \pi I_1^*  \le  
\int_{\RR^2} f(\dd x, \dd v) |x-x_*| \int_{|x_*-y|>|v_*-w|}   \frac{\tilde g(\dd y, \dd w)}{|x_*-y|^2} \\ +
\int_{\RR^2} \tilde f(\dd y, \dd w) |y-y_*| \int_{|x-y_*|>|v-w_*|}  \frac{g(\dd x, \dd v)}{|x-y_*|^2},
\end{multline*}
and we are now in position to apply Lemma~\ref{lem:int_bound_1} as for $I_1$. The last step is the same that for $I_1$ and leads to the bound
\begin{align*}
I_1^*  & \leq \frac{\kappa_\gamma}2 \left( \| \tilde g\|_\gamma \int_{\RR^4} f(\dd x, \dd v) |x-x_*|  +
 \| g \|_\gamma \int_{\RR^4} \tilde f(\dd y, \dd w) |y-y_*| \right) \\
 & \leq \frac{\kappa_\gamma}2 \left( \| \tilde g\|_\gamma Q  + \| g \|_\gamma \tilde Q \right).
\end{align*}
Summing up the two bounds, we eventually get
\begin{equation} \label{eq:bound_I_I}
    I_1 + I_1^* \leq 
\frac{\kappa_\gamma}2 \left( ( \| \tilde f\|_\gamma + \| \tilde g\|_\gamma )
Q  + ( \|  f\|_\gamma + \| g \|_\gamma ) \tilde Q \right).
\end{equation}

\medskip
\textsc{Step 3  : Estimate on $I_2$}

Since $\bar \delta$ depends on the full set of variable $(x,v,y,w)$, we cannot directly use the Fubini Theorem and Lemma~\ref{lem:int_bound_2}. But hopefully, the bound~\eqref{eq:def_delta_bar} satisfied by $\bar \delta$ allows to bypass this difficulty.
We write
\[
D = \{ (x,v,y,w), \; |v-w| \le |x-y| \le |v-w|  + \delta(x,v) + \delta(y,w) \}.
\]
In view of~\eqref{eq:def_delta_bar}, $D \subset D_1 \cup D_2$ with
\begin{align*}
    D_1 & =  \left\{ (x,v,y,w), \; |v-w| \le |x-y| \le |v-w|  + 2  \delta(x,v)   \right\}, \\
    D_2 & =  \left\{ (x,v,y,w), \; |v-w| \le |x-y| \le |v-w|  + 2   \delta(y,w)  \right\},
\end{align*}

Then, using the shortcut $\delta_1= \delta(x,v)$ and $\delta_2=\delta(y,w)$, we can split the integral $I_2$ and use Fubini theorem
\begin{align*}
     2 \pi I_2 &= \iint_{D}  f(\dd x, \dd v) \tilde f(\dd y, \dd w) \frac1{|x-y|} \\ 
     & \leq \iint_{D_1}  f(\dd x, \dd v) \tilde f (\dd y, \dd w) \frac1{|x-y|} + 
     \iint_{D_2}  f(\dd x, \dd v) \tilde f(\dd y, \dd w) \frac1{|x-y|} \\
     & \leq \int_{\RR^4}  f(\dd x, \dd v) \int_{|v-w|\le |x-y| \le |v-w| + \delta_1}  \frac{\tilde f(\dd y, \dd w)}{|x-y|} \\
     & \hspace{2cm} +
     \int_{\RR^4}  \tilde f(\dd y, \dd w) \int_{|v-w|\le |x-y| \le |v-w| +  \delta_2}  \frac{f(\dd x, \dd v)}{|x-y|}.
\end{align*}
The two ``inside'' integrals are now estimated with the help of Lemma~\ref{lem:int_bound_2}
\begin{align}
I_2 & \le   \kappa_\gamma 
\left( \| \tilde f \|_\gamma  \int_{\RR^2}  f(\dd x, \dd v) \delta(x,v) +
     \|  f \|_\gamma  \int_{\RR^2}  \tilde f (\dd y, \dd w) \delta(y,w) \right)  \nonumber \\
     &  =  \sqrt 2 \, \kappa_\gamma \left( \| \tilde f\|_\gamma  Q + \| f\|_\gamma \tilde Q \right),
\label{eq:bound_I_2}
\end{align}
by definition of $Q$ and $\tilde Q$ and the inequalities
\[
Q = \int_{\RR^4} |z - \transpt (z)| f(\dd z) \le
\int_{\RR^4} \left(|x - \transpt^x (z)|+ |v - \transpt^v(z)| \right) f(\dd z) \le \sqrt 2 \, Q,
\]
and a similar relation for $\tilde Q$.

\medskip
\textsc{Step 4 : Estimate on $I_2^*$}

Writing 
\[
D^*  = \{ (x,v,y,w), \; |v_*-w_*| \le |x_*-y_*| \le |v_*-w_*|  + \delta(x,v) + \delta(y,w) \},
\]
$I_2^*$ split in a similar way in two terms
\begin{align*}
     2 \pi I_2^* &= \iint_{D^*}  f(\dd x, \dd v) \tilde f(\dd y, \dd w) \frac1{|x_*-y_*|} \\ 
     & \leq \int_{\RR^4}  f(\dd x, \dd v) \int_{|v_*-w_*|\le |x_*-y_*| \le |v_*-w_*| + \delta_1}  \frac{\tilde f(\dd y, \dd w)}{|x_*-y_*|} \\
     & \hspace{2cm} +
     \int_{\RR^4}  \tilde f(\dd y, \dd w) \int_{|v_*-w_*|\le |x_*-y_*| \le |v_*-w_*| + \delta_2}  \frac{f(\dd x, \dd v)}{|x_*-y_*|}.
\end{align*}
The first  integrals (in the Fubini order) could be bounded using the change of variable 
$(y,w) \to \transpt_t(y,w) = (y_*,w_*)$ and $(x,v) \to \transpt_t(x,v)=(x_*,v_*)$ respectively. Then,
\begin{align*}
     2 \pi I_2^*     & \leq \int_{\RR^4}  f(\dd x, \dd v) \int_{|v_*-w|\le |x_*-y| \le |v_*-w| +  \delta_1}  \frac{\tilde g(\dd y, \dd w)}{|x_*-y|} \\
     & \hspace{2cm} +
     \int_{\RR^4}  \tilde f(\dd y, \dd w) \int_{|v-w_*|\le |x-y_*| \le |v-w_*| +  \delta_2} \frac{g(\dd x, \dd v)}{|x-y_*|}.
\end{align*}
Then as for $I_2$  first simple integrals are now estimated with the help of Lemma~\ref{lem:int_bound_2}, and we end up very similarly with
\begin{equation} \label{eq:bound_I_2*}
I_2^* \leq \sqrt 2 \,  \kappa_\gamma \left( \| \tilde g \|_\gamma  Q + \| g \|_\gamma \tilde Q \right)
\end{equation}


\medskip
\textsc{Step 5 :  Gathering the estimates}

Summing the three previous estimates~\eqref{eq:bound_I_I},\eqref{eq:bound_I_2},\eqref{eq:bound_I_2*}, we obtain the final estimate
\[
	I  \leq \frac3{\sqrt{2}} \kappa_\gamma \left( 
	\bigl(\| \tilde f \|_{\gamma}  +\|g \tilde \|_{\gamma} \bigr) Q +
    \bigl(\|  f \|_{\gamma}  +\|g  \|_{\gamma} \bigr) \tilde Q	\right).
\]
\end{proof}

%
%
\section{A new proof of the existence of solutions}\label{existence-of-weak-solutions}

In this section, we prove Theorem~\ref{theorem-existence}, namely the existence of global weak solutions to~\eqref{limit-equation}. Our proof is quite different from the one in~\autocite{BOS2016}, which is  the consequence of a physical  derivation of the model using compactness methods. Here we start from an ad-hoc mollification of the force kernel, not physically motivated, and provide quantitative stability estimates in the Wasserstein metric between solutions to the mollified equation and solutions to the limit one. 
The estimates proved here are quite similar to the stability estimate~\eqref{stability-estimate} for the limit model proved in the previous section: the novelty is to treat the fact that we use kernels with different mollification parameters, and that we need to carefully estimate the (small) errors that the mollification introduce.

First, we will precise the mollification of the kernel that we shall use, and construct unique and regular solutions to a mollified equation. We will then show that when the mollification parameter goes to zero, the sequence of solutions is indeed a Cauchy sequence for the Wasserstein metric, and eventually its limit will be shown to be a weak solution to~\eqref{limit-equation}.

\subsection{Mollification of the kernels}

We assume that $f_0 \in \proba_1(\RR^4) \cap L^\infty_\gamma(\RR^4)$, for a $\gamma >2$.
We start by smoothing the interaction potential $K$ defined in~\eqref{eq:def_pot} and its derivative $J$ defined in~\eqref{eq:def_J}.  We introduce a family of mollifier
\begin{align} \label{def:mollifier}
\forall\, (x,v) \in \RR^4, \; 	\chi_{\eps}(x,v) = \frac{1}{\eps^4} \chi \left ( \frac{x}{\eps} \right ) \chi \left ( \frac{v}{\eps} \right ),
\end{align}
built with a function $\chi$ such that
\begin{equation*}
	\chi \in \contcomp^{\infty}(\RR^2), \quad
	\supp(\chi) \subset \ball(0,1), \quad 
	\chi \geq 0, \quad 
	\int \chi = 1.
\end{equation*}

Given $\eps>0$, we define the mollified potential $K_\eps = \chi_\eps * K$ and kernel $J_\eps$ 
\begin{equation}
J_\eps =  \chi_\eps * J, \quad \text{i.e. }
J_\eps(x,v) = \frac{1}{\eps^4} \int_{\RR^4}  \chi \left ( \frac{x-y}{\eps} \right ) \chi \left ( \frac{v-w}{\eps} \right ) J(y,w)\dd y \dd w.  \label{def:Jeps}
\end{equation}
%
The following lemma shows that the mollified kernels are bounded and Lipschitz continuous.
\begin{lemma}\label{bounded-lipschitz-mollifiers}
For any $\eps > 0$ 
\[
	\left \| J_{\eps} \right \|_{\infty}  \leq \frac{\|\chi\|_{\infty} \pi}{\eps}, \qquad 
	\left \| \nabla J_{\eps} \right \|_{\infty}  \leq \frac{\| \nabla \chi\|_{\infty} \pi }{\eps^2}.
\]
\end{lemma}
\begin{proof}
	By definition,
	\begin{align*}
		\left | J_{\eps}(x,v) \right | &= \left | \int_{\ball(0,1)^2} \chi(\dd y, \dd w) J(x-\eps y, v-\eps w) \right |
		\leq \frac1{2\pi} \int_{\ball(0,1)^2} \frac{\chi(\dd y, \dd w)}{|x-\eps y|} \\
		&\leq \frac{\|\chi\|_{\infty} }{2 \pi \eps}  \int_{\ball(0,1)^2} \frac{\dd y \dd w}{\left | \frac{x}{\eps}-y \right |} = \frac{\|\chi\|_{\infty} }{2 \eps} \int_{\ball\left ( \frac{x}{\eps},1 \right )} \frac{\dd y}{|y|} 
		\leq \frac{\|\chi\|_{\infty} }{2 \eps} \int_{\ball(0,1)} \frac{\dd y}{|y|} = \frac{\|\chi\|_{\infty} \pi  }{\eps} ,
	\end{align*}
	where we have used a kind of rearrangement inequality : the integral on a radial function on a unit ball is maximal when the ball is centered at the origin
	\[
		\sup_{z \in \RR^2} \int_{\ball(z,1)} \frac{\dd y}{|y|} = \int_{\ball(0,1)} \frac{\dd y}{|y|} = 2 \pi.
	\]
	Using $\nabla_x \left ( \chi_{\eps} \conv J \right ) = (\nabla_x \chi_{\eps}) \conv J$ and $\nabla_v \left ( \chi_{\eps} \conv J \right ) = (\nabla_v \chi_{\eps}) \conv J$, and doing the same computation as above gives the bound for $\nabla J_{\eps}$, with an extra power of $\eps$ coming from the derivation of $\chi_\eps$.
\end{proof}

\subsection{Mollification of densities}

We will also need to mollify the density for the control of the smoothed field. The above lemma  control the Wasserstein distance of order one between a probability distribution and its mollification. 
We shall use it during the computation of the fields later. 

\begin{lemma}\label{wass-dist-initial-regularized}
Let $f \in \mathcal P(\RR^4)$. If $f^{\eps}$ denotes $f \conv \chi_{\eps}$, then for $ \eps>0$
\[
\wass_1(f,f^{\eps}) \leq 2 \eps.
\]
\end{lemma}
\begin{proof}
    We will use an explicit transference plan. Using a transport map seems more difficult because we are not aware of any explicit transport map between $f$ and $f^\eps$, even if $f \in L^1(\RR^4)$. We define $\pi^{\eps} \in \mathcal{P}(\RR^4 \times \RR^4)$ as
	\[
	\pi^{\eps}( \dd x,\dd v,\dd y,\dd w) = f(\dd x,\dd v) \chi_{\eps}(y-x,w-v) \,\dd y \dd w. 
	\]
	It can be checked that $\pi^{\eps}$ is a transference plan, with first marginals $f$ and second  $f^{\eps}$. This implies that
	\begin{align*}
		\wass_1(f,f^{\eps}) &\leq \int_{\RR^4 \times \RR^4} |(x,v)-(y,w)| \, \dd \pi^{\eps}(x,v,y,w) \\
		&= \int_{\RR^4} f(\dd x, \dd v)  \int_{\RR^4} |(y-x,w-v)| \ \chi_{\eps}(y-x,w-v) \, \dd y \dd w \\
		&= \int_{\RR^4} f(\dd x \dd v)  \int_{\RR^4} |(y,w)| \ \chi_{\eps}(y,w) \, \dd y \dd w \\
		&= \eps \int_{\ball(0,1)^2} \bigl( |y|  + |w| \bigr) \ \chi(y) \chi(w) \, \dd y \dd w
		\leq 2 \eps,
	\end{align*}
	where we have used that $f$ is a probability (its total weight is one) on the third line and the definition~\eqref{def:mollifier} of $\chi_\eps$.
\end{proof}

We shall also need a control on the $L^\infty_\gamma$ norm after convolution. The appropriate property is stated in the next Lemma.

\begin{lemma}\label{bound-f0-eps}
	Let $\gamma >0$, $f \in L^\infty_\gamma$, and $\eps >0$. The following inequality holds
	\[
		\|f \conv \chi_{\eps}\|_{\gamma} \leq (1+\eps)^{2\gamma} \|f \|_{\gamma}.
	\]
\end{lemma}
\begin{proof}
	With the help of inequality~\eqref{ineq-x-x'}, it comes :
	\begin{align*}
(f \conv \chi_{\eps})(x,v) &= \int_{\ball(0,\eps)} \chi_{\eps}(y,w) f(x-y,v-w) \, \dd y \dd w \\
		&\leq \int_{\ball(0,\eps)} \chi_{\eps}(y,w) \frac{\|f\|_{\gamma}}{(1+|x-y|)^{\gamma}(1+|v-w|)^{\gamma}} \, \dd y \dd w \\
		&\leq \frac{\|f\|_{\gamma}}{(1+|x|)^{\gamma}(1+|v|)^{\gamma}} \int_{\ball(0,\eps)} \chi_{\eps}(y,w) (1+|y|)^{\gamma}(1+|w|)^{\gamma} \, \dd y \dd w.
	\end{align*}
	Then
	\[
		(f\conv \chi_{\eps})(x,v) \leq \frac{(1+\eps)^{2\gamma}\|f\|_{\gamma}}{(1+|x|)^{\gamma}(1+|v|)^{\gamma}}.
	\]
	Multiplying by~$(1+|x|)^{\gamma}(1+|v|)^{\gamma}$ and taking the supremum on $(x,v)$, we obtain the statement of the lemma.
\end{proof}

\subsection{The mollified equation}

The mollified fields $\vit_{\eps}[f]$ et $\acc_{\eps}[f]$ are then defined very similarly to $\vit[f]$ in~\eqref{vit-expr} and $\acc[f]$ in~\eqref{acc-expr} :
\begin{equation}  \label{regularized-velocity-field} 
	\vit_{\eps}[f] = J_{\eps} \conv f, \qquad 
	\acc_{\eps}[f] = (J_{\eps} \circ S ) \conv f. 
\end{equation}
where we recall that $S$ is the permutation of the variables $(x,v)$. Using that $J_\eps = J \ast \chi_\eps $ and the commutativity of convolutions, we may also write
\begin{equation} \label{regularized-velocity-field-bis}
	\vit_{\eps}[f] = J \conv (f \conv \chi_\eps), \qquad 
	\acc_{\eps}[f] = (J \circ S ) \conv (f \conv \chi_\eps). 
\end{equation}
For those mollified fields, the bound of Proposition~\ref{bounded-fields} holds uniformly in~$\eps$.
\begin{lemma}\label{bound-fields-eps}
	Let $f\in L^1(\RR^4)\cap L^{\infty}(\RR^4)$ and the fields $\vit_{\eps}[f]$ and $\acc_{\eps}[f]$ be defined by~\eqref{regularized-velocity-field}. Then 
\[
	\|\vit_{\eps}[f] \|_{\infty} + \|\acc_{\eps}[f] \|_{\infty} \leq c \, \|f\|_1^{\frac34} \|f\|_{\infty}^{\frac14},
	\quad \text{with  } c = \frac{2^{\frac54}}{3 \sqrt \pi}.
\]
\end{lemma}
\begin{proof}
    Using the equalities~\eqref{regularized-velocity-field-bis} and applying Proposition~\ref{bounded-fields} we get that 
    \[ \left\| \vit_{\eps}[f] \right\|_\infty \le c \, \|f \conv \chi_\eps \|_1^{\frac34} \|f \conv \chi_\eps \|_{\infty}^{\frac14}.
    \]
	And the infinite norm decreases under convolution $\| \chi_\eps \conv f \|_\infty \le  \| \chi \|_1 \| f \|_\infty = \| f \|_\infty $ and so does the $L^1$-norm. This concludes the proof.
\end{proof}

For $\eps>0$, we introduce the mollified equation
\begin{equation}\label{regularized-equation}
	\partial_t f^{\eps} + \vit_{\eps}[f^{\eps}] \cdot \nabla_x f^{\eps} + \acc_{\eps}[f^{\eps}] \cdot \nabla_v f^{\eps} = 0,
\end{equation}
associated to the initial condition
$ \displaystyle f^{\eps}|_{t=0} = f_0$.

The mollified kernel $J_\eps$ being Lipschitz continuous (see lemma~\ref{bounded-lipschitz-mollifiers}), the existence of a unique solution $f^{\eps} \in \cont(\RR_+;\proba_1(\RR^4))$ to~\eqref{regularized-equation} with initial datum $f_0$ is ensured for $\eps>0$. 
We refer to subsection~\ref{subsec:tr_eq} for details. 
$f_t^\eps$ is obtained as the pushforward of 
$f_t^{\eps} = \pushfwd{Z_t^{\eps}} f_0$, where $Z_t^{\eps}$ denotes the flow associated to equation~\eqref{regularized-equation}.

As for smooth solutions of the limit equation~\eqref{limit-equation}, the boundedness of the fields implies the propagation of the $L^\infty_\gamma$-norms along solutions to~\eqref{regularized-equation}. The next lemma is an adaptation of Proposition~\ref{bound-decreasing-f}
\begin{lemma} \label{lem:gamma_eps}
	If $f^{\eps}_t$ is the unique regular solution to the mollified equation~\eqref{regularized-equation} associated to the initial datum $f_0^\eps = f_0 $ qith $f_0 \in L^\infty_\gamma$ for some $\gamma >2$, then the following estimate holds :
	\begin{equation}\label{bound-f-eps}
		\|f^{\eps}_t\|_{\gamma} \leq \left(1+ c \, \|f_0\|_\infty^{\frac14} \|f_0\|_1^{\frac34} t \right)^{2\gamma} \|f_0\|_{\gamma}, \quad t\geq 0,
	\end{equation}
where $c$ is the constant defined in Proposition~\ref{bounded-fields}.
\end{lemma}
\begin{proof}
The proof is very similar to the one of Proposition~\ref{bound-decreasing-f}. The only necessary adaptation is to replace the use of the bound of Proposition~\ref{bounded-fields} in the limit case ($\eps=0$), by the one given by Lemma~\ref{bound-fields-eps}.

\end{proof}

\subsection{The stability estimate in the mollified setting}

The sequence $(f^{\eps})_{\eps>0}$ is hoped to have a limit when $\eps \to 0$. To this purpose, it will be proved that it is a Cauchy sequence by estimating the Wasserstein distance between $f^{\eps}$ and $f^{\eps'}$ for $\eps$ and $\eps'$ positive.
This is done by using again the techniques developed in the previous section, precisely Proposition~\ref{prop:key_estim}.
The control on the Wasserstein distance is in fact obtained by a control on the associated trajectories. That later control allows in fact to pass to the limit also on the trajectories. Thanks to this, we show that the limit $f$ of the sequence $(f_\eps)_{\eps >0}$ is in fact a solution of the expected limit equation.

\medskip
\textsc{Step 1 : Construction of the coupling and first estimate}

We denote $z \to Z_t^{\eps}(z) = (X_t^{\eps}(z),V_t^{\eps}(z))$ for $\eps \ge 0$  the flow of the characteristics associated to $f^{\eps}_t$ at any time $t\geq 0$, having the value $z= (x,v)$ at initial time. They satisfy
\begin{equation} \label{eq:def_Tepseps}
	\frac{\dd}{\dd t} Z_t^{\eps}(z) = ( \vit_{\eps}[f_t^{\eps}], \acc_{\eps}[f_t^{\eps}] )(Z_t^{\eps}(z)), \quad Z_0^{\eps}(z)=z. 
\end{equation}
$Z_t^{\eps}$ is smooth, invertible and transports $f_0$ onto $f_t^{\eps}$ : $\pushfwd{Z_t^\eps} f_0 = f_t^\eps$.
For $\eps, \eps'>0$, a transport map $\transpt_t^{\eps,\eps'}$, transporting $f_t^{\eps}$ onto $f_t^{\eps'}$, can be built by setting
\[
	\transpt_t^{\eps,\eps'} = Z_t^{\eps'} \circ (Z_t^{\eps})^{-1}.
\]
This is illustrated on figure~\ref{transport-regularized-equation-two-solutions}.
\begin{figure}[b]
	\centering
	\begin{tikzpicture}
	\node at (0,0) {$f_t^{\eps}$};
	\node at (2,0) {$f_t^{\eps'}$};
	\node at (1,2) {$f_0$};
	\draw[->] (0.3,0) -- (1.7,0);
	\draw[->] (0,0.3) -- (.9,1.7);
	\draw[->] (1.1,1.7) -- (2,0.3);
	\node[below] at (1,0) {$\transpt_t^{\eps,\eps'}$}; 
	\node[left] at (0.5,1) {$(Z_t^{\eps})^{-1}$};
	\node[right] at (1.6,1) {$Z_t^{\eps'}$};
	\end{tikzpicture}
	\caption{Transport between two solutions to~\eqref{regularized-equation}}
	\label{transport-regularized-equation-two-solutions}
\end{figure}
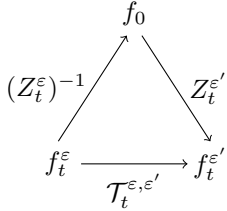
By definition of $\wass_1$ and formula~\eqref{eq:pushforward}, we have
\begin{align*}
	\wass_1(f_t^{\eps},f_t^{\eps'}) &\leq \int_{\RR^4} \left |z-\transpt_t^{\eps,\eps'}(z) \right | \, f_t^{\eps}(\dd z) 
	= \int_{\RR^4} \left |Z^{\eps'}_t(z)-Z^\eps_t(z) \right | \, f_0(\dd z) \\
	\\
& \leq \int_{\RR^4} \left(\sup_{0 \le s \le t } \left |Z^{\eps'}_s(z)-Z^\eps_s(z) \right | \right) \, f_0(\dd z)
	\eqqcolon Q^{\eps,\eps'}(t).
\end{align*}
For brevity, we will write simply  $\transpt_t$ for  $\transpt_t^{\eps,\eps'}$. The notation $\transpt_t=(\transpt_t^x,\transpt_t^v)$ will be used again. 

By definition of $Z^\eps_t$,
\begin{align*}
\sup_{0 \le s \le t } \left |Z^{\eps'}_s(z)-Z^\eps_s(z) \right | &  =
\sup_{0 \le s \le t }  \left |  \int_0^s
\left[ ( \vit_{\eps}[f_u^{\eps}], 
\acc_{\eps}[f_u^{\eps}] )(Z_u^{\eps}(z)) -
( \vit_{\eps'}[f_u^{\eps'}], 
\acc_{\eps'}[f_u^{\eps'}] )(Z_u^{\eps'}(z)) \right] \, \dd u  \right | \\
&  \le 
 \int_0^t  
 \left | 
( \vit_{\eps}[f_u^{\eps}], 
\acc_{\eps}[f_u^{\eps}] )(Z_u^{\eps}(z)) -
( \vit_{\eps'}[f_u^{\eps'}], 
\acc_{\eps'}[f_u^{\eps'}] )(Z_u^{\eps'}(z))
\right | \, \dd u.
\end{align*}
Using now the definition~\ref{eq:def_Tepseps} of the mapping~$\transpt_t = \transpt^{\eps,\eps'}_t$.
\begin{multline*}
\sup_{0 \le s \le t } \left |Z^{\eps'}_s(z)-Z^\eps_s(z) \right | 
 \le 
 \int_0^t  
 \left | 
 \vit_{\eps}[f_u^{\eps}](Z_u^{\eps}(z))- \vit_{\eps'}[f_u^{\eps'}](\transpt_u(Z_u^{\eps}(z))) \right| \, \dd u  \\ 
 + \int_0^t \left| 
\acc_{\eps}[f_u^{\eps}] (Z_u^{\eps}(z))  )-
\acc_{\eps'}[f_u^{\eps'}] )(\transpt_u(Z_u^{\eps}(z))) 
\right| \, \dd u.
\end{multline*}
Integrating now in $z$ w.r.t. $f_0$,
\[
Q^{\eps,\eps'}(t)  = \int_{\RR^4} \sup_{0 \le s \le t } \left |Z^{\eps'}_s(z)-Z^\eps_s(z) \right | f_0(\dd z)  \le \int_0^t I^{\eps,\eps'}_u  \dd u  + \int_0^t J^{\eps,\eps'}_u \dd u.
\]
with (using the fact that $f_t = \pushfwd{Z^\eps_t} f_0$)
\begin{align*}
    I^{\eps,\eps'}_u  & = \int_{\RR^4} \left |  \vit_{\eps}[f_u^{\eps}](Z_u^{\eps}(z)) - \vit_{\eps'}[f_t^{\eps'}](\transpt_u(Z_t^{\eps}(z))) \right |  f_0^{\eps}(\dd z) \\
    & = \int_{\RR^4} \left |  \vit_{\eps}[f_u^{\eps}](z) - \vit_{\eps'}[f_u^{\eps'}](\transpt_u(z)) \right |  f_u^{\eps}(\dd z), \\
 \text{and similarly} \quad    J^{\eps,\eps'}_u  
    & = \int_{\RR^4} \left |  \acc_{\eps}[f_u^{\eps}](z) - \acc_{\eps'}[f_u^{\eps'}](\transpt_u(z)) \right |  f_u^{\eps}(\dd z).
\end{align*}

\textsc{Step 2 : The Grönwall estimate}

We denote $\tilde f^\eps_t = f^\eps_t \conv \chi_\eps$ and $\tilde f^{\eps'}_t = f^{\eps'}_t \conv \chi_{\eps'} $. 
We also denote by $\ttranspt_t = \ttranspt_t^{\eps,\eps'}$ an optimal transport map from $\tilde f^\eps_t$ onto $\tilde f^\eps_t$. It  exists for any time since the distribution $\tilde f^\eps_t$ has a density. 
Remark that we have the control
\begin{equation} \label{wass-T-Tprime}
 \int_{\RR^2} \left| z - \ttranspt_t(z)\right| \tilde f^\eps_t(\dd z) = 
 \wass(\tilde f^\eps_t,\tilde f^{\eps'}_t)  \le  Q^{\eps,\eps'}(t) +2(\eps+\eps').
\end{equation}
This because $\wass(\tilde f^\eps_t,\tilde f^{\eps'}_t) \le \wass(\tilde f^\eps_t, f^{\eps}_t) +
\wass(f^\eps_t,f^{\eps'}_t) + \wass(f^{\eps'}_t,\tilde f^{\eps'}_t) \le \wass(f^\eps_t, f^{\eps'}_t) + 2( \eps+ \eps')$, according to the Lemma~\ref{wass-dist-initial-regularized}.

To control $I^{\eps,\eps'}(t)$, we can use Proposition~\ref{prop:key_estim} with $f = f^\eps_t$, $g=f^{\eps'}_t$,
$\tilde f = \tilde f^\eps_t$, $\tilde g = \tilde f^{\eps'}_t$, $\transpt = \transpt_t^{\eps,\eps'}$
$\ttranspt = \ttranspt_t^{\eps,\eps'}$. It comes
\[
I^{\eps,\eps'}(t)  \le
\frac{3 \kappa_\gamma}2 \left( \bigl(\| \tilde f^\eps_t \|_\gamma + \| \tilde f^{\eps'}_t \|_\gamma\bigr)
\int_{\RR^4} f^\eps_t(\dd z) \left|z - \transpt_t(z) \right|
+ \bigl(\|  f^\eps_t \|_\gamma + \|  f^{\eps'}_t \|_\gamma \bigr)
\int_{\RR^4} \tilde f^\eps_t(\dd z) \left|z - \ttranspt_t(z) \right|
\right).
\]
The Lemma~\ref{lem:gamma_eps} also provide the interesting bound
$\|f^{\eps}_t\|_{\gamma} \leq (1+\alpha t)^{2\gamma} \|f_0\|_{\gamma}$ with $\alpha = c \|f_0\|_\infty^{\frac14}$. 
Next by definition of $\tilde f^\eps_t = f_t^\eps \conv \chi_\eps$ and Lemma~\ref{bound-f0-eps}, we have
\[
\| \tilde f^\eps_t \|_\gamma \le (1+ \eps^{2 \gamma}) \| f^\eps_t \|_\gamma
 \le (1+\eps)^{2\gamma}(1+\alpha t)^{2\gamma} \|f_0\|_{\gamma}.
\]
Using these two bounds in the previous estimate, we get
\[
I^{\eps,\eps'}(t)  \le
3 \kappa_\gamma  (1+\eps \vee \eps')^{2\gamma} (1+\alpha t)^{2\gamma} \|f_0\|_{\gamma} \left( 
\int_{\RR^4} f^\eps_t(\dd z) \left|z - \transpt_t(z) \right|
+ \int_{\RR^4} \tilde f^\eps_t(\dd z) \left|z - \ttranspt_t(z) \right|
\right).
\]
Assuming that $\eps, \eps' \le 1$, and using the estimate~\eqref{wass-T-Tprime}, we eventually get for all  $t \ge 0$
\begin{equation} \label{eq:Ieps_final}
I^{\eps,\eps'}(t)  \le \frac{C_\gamma}2 (1+\alpha t)^{2\gamma} 
 \left(  Q^{\eps,\eps'}(t) + \eps + \eps' \right)
 \quad \text{with } \quad  C_\gamma  := 12 \times 2^{2\gamma}\pi \kappa_\gamma   \|f_0\|_{\gamma}.
\end{equation}
By symmetry, the same estimate holds for $J^{\eps,\eps'}(t)$. It leads to the Grönwall estimate
\[
Q^{\eps,\eps'}(t)  \le 
C_\gamma  \int_0^t (1+\alpha u)^{2\gamma}   \left(  Q^{\eps,\eps'}(u) + \eps + \eps' \right) \, \dd u.
\]
It allows to perform the standard Grönwall lemma on the quantity $Q^{\eps,\eps'}(t) + \eps + \eps'$ which has the same derivative than $Q^{\eps,\eps'}(t)$. With the fact that $Q^{\eps,\eps'}(0)=0$, we get
\begin{equation} \label{eq:eps_gron_estim}
W(f^\eps_t,f^{\eps'}_t)  \le Q^{\eps,\eps'}(t)  \le (\eps + \eps') e^{C'_\gamma (1+\alpha t)^{2\gamma +1}}
\quad \text{with } \quad 
C'_\gamma  := \frac{C_\gamma}{2 \gamma \alpha}.
\end{equation}

\subsection[Convergence of the sequence and recognition of the limit]{Convergence of the sequence $(f^\eps)_{\eps > 0}$ and recognition of the limit}

We continue the proof of Theorem~\ref{theorem-existence}. We show thanks to the new stability estimate that the sequence $(f^\eps)_{\eps > 0}$ is in fact a Cauchy sequence, and so is $(Z^\eps_\cdot)_{\eps >0}$ in the appropriate space.

\medskip
\textsc{Step 3 : Convergence of distributions and flows}

Doing $\eps,\eps' \to 0$ shows that $(f^{\eps})_{\eps>0}$ is a Cauchy sequence in $C([0,T],\mathcal{P}_1(\RR^4))$, endowed with the distance
\[
	d(u,v)=\sup_{t \in [0,T]} \wass_1(u_t,v_t).
\]
This space being complete (see~\autocite[theorem~6.18]{VIL2009}), $(f^{\eps})_{\eps>0}$ has a limit $f$ belonging to the space $C([0,T];\mathcal{P}_1(\RR^4))$ for any $T>0$, so $f$ exists globally in time.
Note that doing only $\eps' \to 0$ in~\eqref{eq:eps_gron_estim} gives the following estimate for $\eps \in (0,1)$ 
\begin{equation} \label{eq:eps_lim_wass}
d(f^{\eps},f) = \sup_{t \in [0,T]} \wass_1(f_t^{\eps},f_t) \leq 
\eps  e^{C'_\gamma (1+\alpha T)^{2\gamma +1}}.
\end{equation}

Define $\Omega = \{ f_0 > 0 \}$. For any $T>0$, the space $\mathcal L_T : = L^1\left( \Omega, C([0,T],\RR^4)\right)$ endowed with the distance 
\begin{equation} \label{eq:eps_lim_LL}
d'_T(X,Y) = \int f_0(\dd z ) \left(\sup_{0 \le t \le T} |X_t(z) - Y_t(z) | \right),
\end{equation}
is also a complete metric space\footnote{We do not have any precise reference for that point but its proof could be done as an adaptation of the proof of the completness of a classical $L^1$ space.}.
So there exists a limit $Z \in \mathcal L_T$ to the sequence $(Z^\eps)_{\eps >0}$ as $\eps \to 0$. And letting $\eps' \to 0$  in~\eqref{eq:eps_gron_estim} Fatou's Lemma leads to
\begin{equation} \label{eq:eqs_lim_LL}
d'_T(Z^\eps,Z) = \int_{\RR^4} 
f_0(\dd z) \left( \sup_{0 \le t \le T} |Z^\eps_t(z) - Z_t(z) | \right)
\le \eps  e^{C'_\gamma (1+\alpha T)^{2\gamma +1}}.
\end{equation}

\medskip
\textsc{Step 4 : Compatibility of the two limits :  $f_t = \pushfwd{Z_t} f_0$ and incompressibility.}

Here we will prove that for any positive time $t$, $f_t = \pushfwd{Z_t} f_0$. It is a simple consequence of the fact that this holds true for any $\eps$ for the regularized solutions $f^\eps$ and the regularized characteristic $Z^\eps$, and the convergence we obtain in the two spaces. So according to~\eqref{eq:measure_preserving} for any lipschitz test function $\varphi$ :
\[
\int  \varphi(z) f^\eps_t(\dd z) =  \int_{\RR^4}  \varphi(Z_t^\eps) f_0(\dd z). 
\]
But thanks to~\eqref{eq:eps_lim_wass}, the term of the l.h.s. goes  as $\eps \to 0$ to $\int  \varphi(z) f_t(\dd z)$. In the same time, according to~\eqref{eq:eps_lim_LL}, the r.h.s.  goes to 
$\int_{\RR^4}  \varphi(Z_t) f_0(\dd z)$. 

It is also classical that we can pass to the limit in the estimate~\eqref{bound-f-eps} to get for all $t \ge 0$
\begin{equation} \label{eq:lim_gamma_norm}
\| f_t \|_\gamma \le (1+ \alpha t)^{2 \gamma} \|f_0\|_\gamma
\end{equation}
with the same constant $\gamma$.
Use for instance the duality
\[
\| f\|_\gamma = \sup \left\{ \int_{\RR^4} f \varphi,  \; \varphi \text{ smooth and  } \int_{\RR^4} \frac{\varphi(z) \, \dd z}{(1+ |x|)^\gamma)(1+|v| )^\gamma}   \le 1 \right\}
\]
and pass to the limit in it.

\medskip
\textsc{Step 5 : The limit flow satisfies the expected ODE} 

We will show that for all $x,v \in \Omega_0$, all $t >0$:
\begin{equation} \label{eq:lim_ODE}
X_t(z)  - x - \int_0^t \vit[f_s](X_s(z), V_s(z))\dd s=0, 
\qquad 
V_t(z)  - v - \int_0^t \vit[f_s](V_s(z), V_s(z))\dd s=0.
\end{equation}
We treat only the $x$ part, the second equality may be proved in a very similar way. 
We start from the fact that the flow $Z^\eps_t=(X_t^\eps,V_t^\eps)$ associated to the mollified equation satisfies the ODE written in integral form for all $x,v \in \Omega_0$, all $t >0$:
\[
X^\eps_t(z)  - x - \int_0^t \vit[f_s^\eps](X^\eps_s(z), V^\eps_s(z))\dd s =0, 
\]
We will integrate the equality on $x$ again $f_0$ : 
\[
\int_{\RR^4} f_0(\dd z)
\left|
X^\eps_t(z)  - x - \int_0^t \vit[f_s^\eps](X^\eps_s(z), V^\eps_s(z))\dd s
\right| = 0.
\]
Then using this term as a pivot, we deduce that
\begin{multline} \label{eq:limit_flow}
\int_{\RR^4} f_0(\dd z)
\left|
X_t(z)  - x - \int_0^t \vit[f_s](Z_s(z))\dd s
\right| \le  
\int_{\RR^4} f_0(\dd z) \left| X^\eps_t(z) - X_t(z) \right|
\\ +
\int_0^t \dd s  \int_{\RR^4} f_0(\dd z) \left| \vit[f_s^\eps](Z^\eps_s(z)) - \vit[f_s](Z_s(z)) \right| 
.
\end{multline}
The first term in the r.h.s. goes to zero according to~\eqref{eq:eps_lim_LL} as $\eps \to 0$. The second term is also a estimate of difference in advection field created by different distributions, so we will again apply Proposition~\ref{prop:key_estim}. In fact, using the change of variable $z'= Z^\eps_s(z)$, we write
\[
\int_{\RR^4} f_0(\dd z) \left| \vit[f_s^\eps](Z^\eps_s(z)) - \vit[f_s](Z_s(z)) \right| =
\int_{\RR^4} f_s^\eps (\dd z) \left| \vit[f_s^\eps](z) - \vit[f_s](Z_s \circ (Z^\eps_s)^{-1} (z) ) \right|.
\]
We could apply Proposition~\ref{prop:key_estim} with 
$f = \tilde f = f^\eps_s$, $g = \tilde g  = f_s$, and $\transpt=\ttranspt = Z_s \circ (Z^\eps_s)^{-1}$ which is a transport of $f^\eps_s$ onto $f_s$ according to the previous step.
We obtain then the bound
\begin{align*}
\int_{\RR^4} f_0(\dd z) \left| \vit[f_s^\eps](Z^\eps_s(z)) - \vit[f_s](Z_s(z)) \right|  & \le 
3  \kappa_\gamma 
\left( \| f_s^\eps \|_\gamma +  \| f_s \|_\gamma \right)
\int_{\RR^4} f_s^\eps (\dd z) \left| z - Z_s \circ (Z^\eps_s)^{-1} (z) \right| \\
& \le C  (1+\alpha s)^{2 \gamma} \int_{\RR^4} f_0(\dd z) \left| Z^\eps_s(z) - Z_s(z) \right|,
\end{align*}
where the constant $C$ depends on $\| f_0 \|_\gamma$ but does not depend on $\eps$, provided that $\eps \le 1$ (for instance).  We used in fact Lemma~\ref{lem:gamma_eps} and Lemma~\ref{bound-f0-eps} as in Step 2 to control $\|f_s\|_\gamma$ and $\|f_s^\eps\|_\gamma$ in  term of $\eps$ and $\|f_0\|_\gamma$.
Using the last estimate in~\eqref{eq:limit_flow}, we get
\[
\int_{\RR^4} f_0(\dd z)
\left| X_t(z)  - x - \int_0^t \vit[f_s](Z_s(z))\dd s \right| \le C \frac{
\left(1 + \alpha t \right)^{2 \gamma+1}}{2 \alpha \gamma}
\int_{\RR^4} f_0(\dd z) \left( \sup_{0 \le s \le t} \left| Z^\eps_s(z) - Z_s(z) \right| \right).
\]
From~\eqref{eq:eqs_lim_LL} and since this is valid for any $\eps>0$, the l.h.s. of~\eqref{eq:limit_flow} vanishes.

\medskip
\textsc{ Step 7 : Conclusion of the proof}

We know that $f_t = \pushfwd{Z_t} f_0$, and that the flow family $(Z_t)_{t \ge 0}$ satisfies the expected ODE~\eqref{eq:lim_ODE}.
From the bound~\eqref{eq:eps_lim_LL} on the $L^\infty_\gamma$-norm of $f_t$ and Proposition~\ref{prop-fields-lipschitz}, we know that the field $(\vit[f_s],\acc[f_s])$ is Lipschitz in $(x,v)$. This implies that $f$ is solution to the expected limit equation.

\printbibliography

\end{document}